\documentclass[12pt, reqno]{amsart}
 
 \usepackage{amsmath,amsfonts,amsthm,amssymb,color,hyperref, mathrsfs, graphicx, mathabx,tikz}

\usepackage{framed}
\usepackage{lscape}
\usepackage{array}
\usepackage{calc}
\usepackage[T1]{fontenc}

\usepackage[left=1.1in, right=1.1in, top=1.1in,bottom=1.1in]{geometry} 

\newtheorem{theorem}{Theorem}[section]
\newtheorem{lemma}[theorem]{Lemma}

\theoremstyle{definition}

\theoremstyle{remark}
\newtheorem{remark}{Remark}
\numberwithin{equation}{section}

\newcommand{\R}{\mathbb R}
 \newcommand{\N}{\mathbb N}

\newcommand{\E}{\mathbb E}

 \newcommand{\PP}{\mathbb P}

\begin{document}

\title[Stein's method and a cubic mean field-model]{Stein's method and a cubic mean-field model}

 \date{}

 \author[P.\ Eichelsbacher]{Peter Eichelsbacher} 
 \address{Faculty of Mathematics, Ruhr University Bochum, Germany.}
 \email{peter.eichelsbacher@rub.de}
  

\maketitle

\begin{abstract}  In this paper, we study a mean-field spin model with three- and two-body interactions. In a recent paper \cite{CMO2024}
by Contucci, Mingione and Osabutey, the equilibrium measure for large volumes was shown to have three pure states, two with
opposite magnetization and an unpolarized one with zero magnetization, merging at the critical point. The authors proved a central limit theorem
for the suitably rescaled magnetization. The aim of our paper
is presenting a prove of a central limit theorem for the rescaled magnetization applying the exchangeable pair approach due to Stein.
Moreover we prove (non-uniform) Berry-Esseen bounds, a concentration inequality, Cram\'er-type moderate deviations and a moderate deviations principle
for the suitably rescaled magnetization. Interestingly we analyze Berry-Esseen bounds in case the model-parameters $(K_n,J_n)$ converge to the critical
point $(0,1)$ on lines with different slopes and with a certain speed, and obtain new limiting distributions and thresholds for the speed of convergence. 
\\[0.5cm]
\medskip\noindent {\bf Mathematics Subject Classifications (2020)}: 60E15, 60F05, 60F10, 82B20, 82B26. \newline
\noindent {\bf Keywords:} uniform and non-uniform Berry-Esseen bound, moderate deviations principle, Cram\'er-type moderate deviations,
concentration inequality, Stein's method, exchangeable pairs, Curie-Weiss model, mean-field model, cubic mean-field Ising model, cubic Curie-Weiss model.
\end{abstract}

\allowdisplaybreaks

\section{Introduction}

Stein's method is a powerful tool to prove distributional approximation. One of its advantages is that is
often automatically provides also a rate of convergence and may be applied rather effectively also to classes 
of random variables that are stochastically dependent. Such classes of random variables are in natural way provided 
by spin system in statistical mechanics. The easiest of such models are mean-field models. 
Among them the Curie--Weiss model is well known for exhibiting a number of
properties of real substances, such as spontaneous magnetization or metastability. 

The aim of this paper is to apply Stein's method for exchangeable pairs (see \cite{Stein:1986})
for distributional approximations and thereby prove Berry-Esseen bounds for the suitably rescaled magnetization of a mean-field
model with three- and two-body interaction called {\it Cubic Curie-Weiss model} or {\it Cubic mean-field Ising model}. The interest in this model
is described in \cite{CMO2024}. The interest comes from condensed matter physics (see e.g. \cite{SL:1999}) as well as from complex systems
of socio-technical nature where a three-body interaction entails moving from a graph-theoretical environment of vertices and edges
to a richer hypergraph setting, see e.g. \cite{ACMM:2017}.  The presence of the cubic interaction brings technical difficulties to the analysis of the model,
since now the energy is non-convex. Techniques like the Hubbard-Stratonovich transform or the 
Legendre transform of the logarithm of the moment generating function are prevented by the cubic power. But a 
large deviation principle (LDP) for the magnetization per particle can easily be proved applying Varadhan's lemma. The rate function of this principle
will play an important role.  A similar observation was already obtained in recent work concerning the fluctuations of the rescaled magnetization of the $p$-spin Curie Weiss model, see \cite{MSB:2021} and \cite{MLB:2024}.

In case of the cubic Curie-Weiss model, where a three- and a two-body interaction with strength $K$ and $J$ 
we are interested in fluctuations on any scale different form a law of large numbers or the large deviation principle of the magnetization per particle.
We will ask for moderate deviations, Cram\'er-type moderate deviations and Berry-Esseen bounds for normal and nonnormal approximations.
For an overview of results on the Curie-Weiss models with a pure quadratic interaction and related models, see \cite{Ellis:LargeDeviations},
\cite{Ellis/Newman:1978} and \cite{Ellis/Newman/Rosen:1980}. In \cite{CMO2024}, the authors proved that the central limit theorem holds for a suitably rescaled magnetization, while its violation with the typical quartic behaviour appears at the critical point. The last observation was known before. 

Very impressive is that in \cite{CMO2024} the complete phase diagram in $(K,J)$ of the equilibrium measure was presented. Outside a curve $\gamma$, $K \mapsto \gamma(K)=J$, the rate function of the large
deviation principle has a unique minimum. On the curve there are two global minimizers. This corresponds to the existence of two
different thermodynamic equilibrium phases, a {\it stable paramagnetic state} with the absence of spontaneous magnetic order and
the presence of symmetry between the up and down spin, and a {\it positively polarised state}. The jump between these states
represents a {\it first-order phase transition}. The quadratic Curie-Weiss model has a second-order phase transition in the inverse temperature.
The new behaviour in the cubic model reminds of the Curie-Weiss Potts model, where a first order phase transition
was observed, see \cite{EllisWang:1990} and \cite{EM:2010}.
\medskip

 Let us consider $n$  spins $x = (x_i)_{i \in \N} \in \{ -1, 1 \}^n$ interacting through an Hamiltonian of the form
 \begin{equation} \label{Ham}
 H_n(x) = - \frac{K}{3 n^2} \sum_{i,j,k=1}^n x_i x_j x_k - \frac{J}{2 n} \sum_{i,j=1}^n x_i x_j - h \sum_{i=1}^n x_i,
 \end{equation}
 where $K >0$ and $J \in \R$ are interaction parameters and $h$ represents an external field. We will only consider $h=0$.
 With
 \begin{equation} \label{mag}
 m(x) := m_n(x) := \frac 1n \sum_{i=1}^n x_i \,\, \text{and} \,\, S_n := S_n(x) := \sum_{i=1}^n x_i
 \end{equation}
 denoting the {\it magnetization per particle} and {\it the total magnetization}, respectively, we obtain the mean-field nature of $H_n$ because
 $$
 H_n(x) = -n \biggl( \frac K3 m_n^3(x) + \frac J2 m_n^2(x) \biggr) = - n \biggl( \frac K3 \bigl(\frac{S_n}{n} \bigr)^3 + \frac J2 \bigl( \frac{S_n}{n} \bigr)^2 \biggr).
 $$
The Boltzmann-Gibbs probability measure associated to $H_n$ is defined by
\begin{equation} \label{BG}
\PP_n ((x_i)) := \PP_{n,K,J}((x_i)) := \PP_{n,K,J}  \bigl( (X_i)_i = (x_i)_i \bigr) := \frac{1}{Z_{n,K,J}} \exp \bigl( - H_n(x) \bigr) \prod_{i=1}^n d \varrho(x_i)
\end{equation}
with $\varrho(-1)=\varrho(1) = \frac 12$, or $\varrho = \frac 12 \delta_{-1} + \frac 12 \delta_{1}$.
Since the Hamiltonian is invariant under transformation $ K \mapsto -K$ and $x_i \mapsto -x_i$ for $i=1, \ldots, n$, the choice $K>0$ is without loss.

We like to point out, that with any distribution $\varrho$ of a single spin, choosen to be a probability measure on $\R$ endowed with the Borel $\sigma$-field,
the model can be generalized obviously. The $Z_{n,K,J}$ is the normalizing constant that turns $\PP_n$ into a probability measure and is called the partition function. 

The case $K=0$ and  $\varrho = \frac 12 \delta_{-1} + \frac 12 \delta_{1}$ reduces to the Curie-Weiss model. In \eqref{BG} we set the usual inverse temperature $\beta$
to 1 without loss since it has been reabsorbed in the parameters $J$ and $K$. 
The measures $\PP_n$ is completely determined by the value of the magnetization. The total magnetization is therefore called an order parameter
and its behavior will be studied in this paper.

In the classical model ($K=0$ and $\varrho$ is the symmetric Bernoulli measure), for
$0 < J < 1$, in \cite{Simon/Griffiths:1973} and \cite{Ellis/Newman:1978} the
central limit theorem is proved: 
$ \frac{S_n}{\sqrt{n}} = \sqrt{n} m_n \to N(0, \sigma^2(J))$
in distribution with respect to the Curie--Weiss finite volume Gibbs states
with $\sigma^2(J) = (1-J)^{-1}$. Here $N(0, \sigma^2(J))$ denotes a Gaussian distribution
with mean $0$ and variance $\sigma^2(J)$.
Since for $J = 1$ the variance $\sigma^2(J)$ diverges, the
central limit theorem fails at the critical point. In \cite{Simon/Griffiths:1973} and \cite{Ellis/Newman:1978} it is proved that
for $J = 1$ there exists a random variable $X$ with probability density proportional to $\exp(- \frac{1}{12} x^4)$
such that $\frac{S_n}{n^{3/4}} \to X$ as $n \to \infty$
in distribution with respect to the finite-volume Boltzmann-Gibbs states.
In \cite[Theorem 2.1]{CMO2024}, outside the curve $\gamma$ of values $(K,J)$, where two minimizers of the
rate function coexist, a central limit theorem was proved for the rescaled magnetization. Since at the critical point of the model, the cubic
part of the Hamiltonian disappears and the inverse temperature is 1, the quartic behaviour appears, see \cite{Simon/Griffiths:1973}. 
\medskip

Stein introduced in \cite{Stein:1986} the exchangeable pair approach. Given a random variable $W$, Stein's method
is based on the construction of another variable $W'$ (some coupling) such that the pair $(W, W')$ is exchangeable, i.e. their
joint distribution is symmetric. The approach essentially uses the elementary fact that if $(W,W')$ is an exchangeable
pair, then $\E g(W,W')=0$ for all antisymmetric measurable functions $g(x,y)$ such that the expectation exists.
A theorem of Stein (\cite[Theorem 1, Lecture III]{Stein:1986})
shows that a measure of proximity of $W$ to normality may be provided in terms
of the exchangeable pair, requiring $W-W'$ to be sufficiently small. He assumed the {\it linear regression property}
$$
\E(W-W'|W)=\lambda \, W
$$
for some $0 < \lambda <1$. Stein proved that for any uniformly Lipschitz function $h$,
$| \E h(W) - \E h(Z) | \leq \delta \|h'\|$
with $Z$ denoting a standard normally distributed random variable and
\begin{equation} \label{del1}
\delta = 4 \E \bigg| 1 - \frac{1}{2 \lambda} \E \bigl( (W'-W)^2 | W \bigr) \bigg| + \frac{1}{2 \lambda} \E |W-W'|^3.
\end{equation}
Stein's approach has been successfully applied in many models, see e.g. \cite{Stein:1986} or
\cite{DiaconisStein:2004} and references therein. In \cite{Rinott/Rotar:1997}, Rinott and Rotar extended the range of application
 by replacing the linear regression property by a weaker condition, which is
\begin{equation} \label{exid1b}
\E(W-W'|W) = \lambda \, W + R
\end{equation}
with $R=R(W)$ being some random reminder term. With Theorem 1.2 in \cite{Rinott/Rotar:1997}, in the weaker case is holds that
for any uniformly Lipschitz function $h$,
$
| \E h(W) - \E h(Z) | \leq \delta' \|h'\|
$
with
\begin{equation} \label{del2}
\delta' = 4 \E \bigg| 1 - \frac{1}{2 \lambda} \E \bigl( (W'-W)^2 | W \bigr) \bigg| + \frac{1}{2 \lambda} \E |W-W'|^3 + 19 \frac{\sqrt{\E R^2}}{\lambda}.
\end{equation}
\medskip

The paper is organized as follows: 
In Section 2 we cite several recent plug-in theorems within the theory of the exchangeable pair approach. Due to Berry-Esseen bounds, we cite results
form Eichelsbacher and L\"owe \cite{ELStein} and from Shao and Zhang \cite{ShaoZhang:2019}, where in case of an exchangeable pair satisfying \eqref{exid1b} and with $|W-W'|$ being bounded, one can observe quite simple Berry-Esseen bounds. We collect recent results 
of non-uniform Berry-Esseen bounds in the context of the exchangeable pair approach due to Liu, Li, Wang and Chen, \cite{LLWC21}. Moreover
we cite a result due to Chatterjee on Stein's method and concentration inequalities, \cite{Chatterjee:2007}, and present a theorem how to prove a Cram\'er-type moderate deviation result in the context of exchangeable pairs, originally proved for normal approximation by Chen, Fang and
Shao in \cite{CFS13} and for some nonnormal distributions in \cite{ShaoZhangZhang:2021}.
In Section 3 we state some important results on the cubic mean-field model proved in \cite{CMO2024}.
In Section 4 we state and prove (non-uniform) Berry-Esseen bounds for the normal and nonnormal approximation for the suitably rescaled magnetization in the cubic mean-field model. Moreover we prove Berry-Esseen bounds in case the model-parameters $(K_n,J_n)$ converge to the critical
point $(0,1)$ on lines with different slopes and with a certain speed, and obtain new limiting distributions and a threshold for the speed of convergence.
In Section 5 we prove a concentration inequality for the total magnetization, in Section 6 we prove a Cram\'er-type moderate deviations for the total magnetization and obtain a moderate deviations principles. 
\bigskip

\section{Exchangeable pair approach for distributional approximations and moderate deviations}

Given two real-valued random variables $X$ and $Y$ defined on a common probability space, we denote the Kolmogorov distance of the distributions of $X$ and $Y$ by
$$
d_{\rm{K}}(X,Y) := \sup_{z \in \R} | P(X \leq z) - P(Y \leq z)|.
$$

\subsection{Normal approximation}

Considering an exchangeable pair $(W,W')$ with $|W - W'| \leq A$ with $A>0$ the following result 
in \cite[Theorem 2.6]{ELStein} improves \cite[Theorem 1.2]{Rinott/Rotar:1997} with respect to the Berry-Esseen constants:

\begin{theorem} \label{ourmain}
Let $(W, W')$ be an exchangeable pair of real-valued random variables such that
$$
\E(W-W' |W) = \lambda\, W +R
$$
for some random variable $R = R(W)$ and with $0 < \lambda <1$. Assume that $|W - W'| \leq A$ for $A>0$ and $\E(W^2) \leq 1$. Let $Z$ be a random variable
with standard normal distribution.
Then we have
\begin{eqnarray} \label{mainbound}
d_{\rm{K}}(W, Z) & \leq & \sqrt{ \E \biggl( 1 - \frac{1}{2 \lambda}  \E[(W'-W)^2 |W] \biggr)^2}  \nonumber \\
& & + \biggl( \frac{\sqrt{2 \pi}}{4} + 1.5 A \biggr) \frac{\sqrt{\E(R^2)}}{\lambda}  + \frac{0.41 A^3}{\lambda} +  1.5 A .
\end{eqnarray}
\end{theorem}

\begin{remark}
In Theorem \ref{ourmain}, we assumed $\E(W^2) \leq 1$. Alternatively, let us assume that $\E(W^2)$ is finite. Then the
proof of Theorem \ref{ourmain} shows, that the third and the fourth summand of the bound change to
$
\frac{A^3}{\lambda} \bigl( \frac{\sqrt{2 \pi}}{16} + \frac{\sqrt{ \E(W^2)}}{4} \bigr) + 1.5 A \, \E(|W|)$.
\end{remark}

An important step in the development of Stein's method was made in \cite{ShaoZhang:2019}.  If there is given an exchangeable pair
satisfying \eqref{exid1b} for some constant $\lambda \in (0,1)$ and $\Delta := W-W'$. Then
$$
d_{\rm{K}}(W, Z) \leq \E \bigg| 1 - \frac{1}{2 \lambda} \E (\Delta^2 | W) \bigg| + \frac{\E |R|}{\lambda} + \frac{1}{\lambda} \E \big| \E( \Delta \, \Delta^* | W) \big|,
$$
where $\Delta^* := \Delta^*(W,W')$ is any random variable satisfying $ \Delta^*(W,W') =  \Delta^*(W',W)$ and $ \Delta^* \geq |\Delta|$.

The following consequence of the last result is stated in \cite[Corollary 2.1]{ShaoZhang:2019}, which will be the result which will be applied in this paper:

\begin{theorem} \label{CLT-1}
If there is given an exchangeable pair satisfying \eqref{exid1b} for some constant $\lambda \in (0,1)$ and $\Delta := W-W'$.  If moreover 
$|W - W'| \leq A$ with $A>0$ and $\E |W| \leq 2$,  then
\begin{equation} \label{SZ19main}
d_{\rm{K}}(W, Z) \leq \E \bigg| 1 - \frac{1}{2 \lambda} \E (\Delta^2 | W) \bigg| + \frac{\E |R|}{\lambda} + 3 A.
\end{equation}
\end{theorem}

Notice that the term $\frac{A^3}{\lambda}$ in \eqref{mainbound} does not appear in the preceding theorem at the cost of assuming $\E |W| \leq 2$,
which is easily satisfied. 
\medskip

In recent years Stein's method was also successfully applied to prove {\it non-uniform} Berry-Esseen bounds. In case of sums of
independent or locally dependent random variables, Chen and Shao proved such non-uniform bounds with the help of concentration
inequalities in \cite{CS01} and \cite{CS04}. Recently Liu, Li, Wang and Chen obtained in \cite{LLWC21} a plug-in theorem for non-uniform Berry-Esseen bounds
in the case of the exchangeable pair approach. If there is given an exchangeable pair $(W,W')$ satisfying \eqref{exid1b} for some constant $\lambda \in (0,1)$ and 
$\Delta := W-W'$ and $\E W^2 < \infty$. Then for any $z \in \R$, one is able to obtain
\begin{equation} \label{non-uniform1}
| \PP(W \leq z) - \PP(Z \leq z) | \leq \frac{c}{1 + |z|} \biggl( 
 \E \bigg| 1 - \frac{1}{2 \lambda} \E (\Delta^2 | W) \bigg| + \frac{\E |R|}{\lambda} + \frac{1}{\lambda} \E \big| \E( \Delta \, \Delta^* | W) \big| \biggr),
\end{equation}
where $\Delta^* := \Delta^*(W,W')$ is any random variable satisfying $ \Delta^*(W,W') =  \Delta^*(W',W)$ and $ \Delta^* \geq |\Delta|$.
In the bounded case we will apply the following result, which is an obvious combination of \cite[Theorem 2.1]{LLWC21} and 
Theorem \ref{CLT-1}. 

\begin{theorem} \label{CLT-2}
If there is given an exchangeable pair satisfying \eqref{exid1b} for some constant $\lambda \in (0,1)$ and $\Delta := W-W'$.  If moreover 
$|W - W'| \leq A$ and $\E W^2 \leq 2$,  then
\begin{equation} \label{non-uniform2}
| \PP(W \leq z) - \PP(Z \leq z) | \leq \frac{c}{1 + |z|}  \biggl( \E \bigg| 1 - \frac{1}{2 \lambda} \E (\Delta^2 | W) \bigg| + \frac{\E |R|}{\lambda} + 3 A \biggr),
\end{equation}
where $c$ is a constant which depends only on $\E W^2$.
\end{theorem}

\begin{remark} \label{remark2}
In case the random variable $R$ in \eqref{exid1b} does not appear, the result in \cite{LLWC21} can simply be improved. Whenever $W$ satisfies
the assumption $\E \bigl( W^{2p} \bigr) < \infty$ for $p \geq 1$ being an integer, then \eqref{non-uniform1} can be replaced by
$$
| \PP(W \leq z) - \PP(Z \leq z) | \leq \frac{c}{1 + |z|^p} \biggl( 
 \E \bigg| 1 - \frac{1}{2 \lambda} \E (\Delta^2 | W) \bigg| + \frac{1}{\lambda} \E \big| \E( \Delta \, \Delta^* | W) \big| \biggr).
$$
The reason is quite simple. Considering the proof of Theorem 2.1 in \cite{LLWC21}, with the help of Markov's inequality with respect to the
$2p$-th moment, is is possible to improve (3.7) and (3.12) in \cite{LLWC21} to
$$
\sqrt{ \E |f_z'(W)|^2} \leq \frac{c_1}{1 + |z|^p} \,\, \text{and} \,\, \sqrt{ \E |W f_z(W)|^2} \leq \frac{c_2}{1 + |z|^p},
$$
with constants $c_1$ and $c_2$ only depending on $\E (W^{2p})$.  Here, $f_z$ with fixed $z \in \R$ denotes the solution of the Stein equation
$$
f_z'(x) - x f_z(x) = 1_{\{x \leq z\}} - P(Z \leq z), \,\, x \in \R.
$$
In case the reminder $R$ appears, one has to bound 
$\sqrt{ \E | f_z(W) |^2}$, but its bound $\frac{1}{|z|}$ cannot be improved, see \cite{ShaoZhangZhang:2021}, Lemma 4.1 and it's proof.
\end{remark}

\subsection{Nonnormal approximation} \label{section2.2}

\medskip
Interesting enough, in \cite{ShaoZhang:2019} Berry-Esseen bounds for nonnormal approximation are considered. With respect to the
exchangeable pair approach, this is a further development of results in \cite{ChatterjeeShao:2011} and \cite{ELStein}. Let $W$ be random variable
and let $(W,W')$ be an exchangeable pair satisfying
\begin{equation} \label{nonnormal}
\E (W - W'|W) = \lambda g(W) + R,
\end{equation}
where $g(x) := g_k(x) = a_1 x^{2k-1} + a_2 x^{2k+1}$ for an integer $k \geq 1$, $a_1, a_2 \in \R$, $\lambda \in(0,1)$ and $R$ is a random variable. Consider $G_k(y) := \int_{-\infty}^y g_k(t) dt$ and $p_k(y) :=  
c_k \exp( - G_k(y))$ and $c_k$ is the constant so that $\int_{\mathbb{R}} p_k(y) dy =1$. Let $Y_k$ be a random variable with the probability density $p_k$, and let 
$\Delta := W -W'$. Interesting enough \cite[Theorem 2.2]{ShaoZhang:2019} reads as
$$
d_{\rm{K}}(W, Y_k) \leq \E \bigg| 1 - \frac{1}{2 \lambda} \E (\Delta^2 | W) \bigg| + \frac{\E |R|}{c_k \lambda} + \frac{1}{\lambda} \E \big| \E( \Delta \, \Delta^* | W) \big|,
$$
where again $\Delta^* := \Delta^*(W,W')$ is any random variable satisfying $ \Delta^*(W,W') =  \Delta^*(W',W)$ and $ \Delta^* \geq |\Delta|$.

We will apply the following consequence of the last bound:

\begin{theorem} \label{CLT-3}
If there is given an exchangeable pair satisfying \eqref{nonnormal} for some constant $\lambda \in (0,1)$, $g_k(x) = a_1 x^{2k-1}+ a_2 x^{2k+1}$ for some integer $k \geq 1$, $a_1, a_2 \in \R$ and $\Delta := W-W'$.  If moreover 
$|W - W'| \leq A$ and $\E |W| \leq 2$,  then
\begin{equation} \label{SZ19mainnonnormal}
d_{\rm{K}}(W, Y_k) \leq \E \bigg| 1 - \frac{1}{2 \lambda} \E (\Delta^2 | W) \bigg| + \frac{\E |R|}{c_k \lambda} + 3 A.
\end{equation}
\end{theorem}

Also a non-uniform Berry-Esseen bound is available, see \cite{LLWC21}. A combination of Theorem 2.1 in \cite{LLWC21} with Theorem 2.2 and Corollary
2.1 in \cite{ShaoZhang:2019} leads to:

\begin{theorem} \label{CLT-4}
If there is given an exchangeable pair satisfying \eqref{nonnormal} for some constant $\lambda \in (0,1)$, $g_k(x) = a_1 x^{2k-1}+ a_2 x^{2k+1}$ for some integer $k \geq 1$,
$a_1, a_2 \in \R$
and $\Delta := W-W'$.  If moreover 
$|W - W'| \leq A$ and $\E W^2 \leq 2$,  then
\begin{equation} \label{non-uniform4}
| \PP(W \leq z) - \PP(Y_k \leq z) | \leq \frac{c}{1 + |g(z)|}  \biggl(  \E \bigg| 1 - \frac{1}{2 \lambda} \E (\Delta^2 | W) \bigg| + \frac{\E |R|}{c_k \lambda} + 3 A \biggr).
\end{equation}
\end{theorem}

\begin{remark}
As pointed out in Remark \ref{remark2}, in case the random variable $R$ in \eqref{nonnormal} does not appear, the result in \cite{LLWC21} can simply be improved. Whenever $W$ satisfies the assumption $\E \bigl( W^{2p} \bigr) < \infty$, with $p\geq 1$ an integer, then \eqref{non-uniform4} can be replaced by
$$
| \PP(W \leq z) - \PP(Y_k \leq z) | \leq \frac{c}{1 + |g(z)|^p} \biggl( 
 \E \bigg| 1 - \frac{1}{2 \lambda} \E (\Delta^2 | W) \bigg| + \frac{1}{\lambda} \E \big| \E( \Delta \, \Delta^* | W) \big| \biggr).
$$
Considering the proof of Theorem 2.1 in \cite{LLWC21}, with the help of Markov's inequality with respect to the
$2p$-th moment, is is possible to improve (3.7) and (3.12) in \cite{LLWC21} to
$$
\sqrt{ \E |f_z'(W)|^2} \leq \frac{c_1}{1 + |g(z)|^p} \,\, \text{and} \,\, \sqrt{ \E |g(W) f_z(W)|^2} \leq \frac{c_2}{1 + |g(z)|^p},
$$
with constants $c_1$ and $c_2$ only depending on $\E (W^{2p})$.  
\end{remark}

\subsection{Concentration inequalities and Cram\'er-type moderate deviations via exchangeable pairs} \label{section2.3}

With the exchangeable pair approach it is also possible to prove a concentration inequality for the magnetization in the cubic mean-field model.
Therefore we will apply the following result due to Chatterjee, \cite[Theorem 1.5]{Chatterjee:2007}.

\begin{theorem} \label{Chat}
Let $(X,X')$ be an exchangeable pair of real-valued random variables, let $f: \R \to \R$ and $F: \R \times \R \to \R$ are square-integrable functions such that
$F$ is antisymmetric and $\E(F(X,X') |X) = f(X)$ almost surely. Let 
$$
\Delta(X) := \frac 12 \E \bigl( |(f(X)-f(X'))F(X,X') | \big| X \bigr).
$$
Assume that $\E(e^{\theta f(X)} | F(X,X')|) < \infty$ for all $\theta$. If there exists nonnegative constants $B$ and $C$ such that $\Delta(X) \leq B \, f(X) + C$,
then for any $t>0$,
$$
\PP \bigl( f(X) \geq t \bigr) \leq \exp \biggl( - \frac{t^2}{2C + 2 B t} \biggr) \,\, \text{and} \,\, \PP \bigl( f(X) \leq - t \bigr) \leq \exp \biggl( - \frac{t^2}{2C} \biggr).
$$
\end{theorem}

\medskip
Stein's method was also applied to obtain a general Cram\'er-type moderate deviation results, starting with \cite{CFS13} in the context
of normal approximation and \cite{ShaoZhangZhang:2021} in the context of nonnormal approximation. Moderate deviations date back to Cram\'er
in 1938, proving that for independent and identically distributed random variables $X_1, \ldots, X_n$ with $\E(X_1)=0$ and variance 1 such that
$\E \bigl( e^{t_0 |X_1|} \bigr) \leq c < \infty$ for some $t_0 >0$, it follows that
$$
\frac{\PP (W_n > x)}{\PP (Z > x)} = 1 + {\mathcal O}(1) (1+ x^3) / \sqrt{n}
$$
for $0 \leq x \leq a_0 n^{1/6}$, where $W_n := (X_1 + \cdots + X_n) / \sqrt{n}$, $Z$ is standardnormal distributed, $a_0 >0$ depends on $c$ and $t_0$ and 
${\mathcal O}(1)$ is bounded by a constant depending on $c$ and $t_0$. The range  $0 \leq x \leq a_0 n^{1/6}$ and the order of the error term
${\mathcal O}(1) (1+ x^3) / \sqrt{n}$ are optimal.

In the context of the exchangeable
pair approach, Theorem 3.1 from \cite{CFS13} reads as following:

\begin{theorem} \label{CramerTheo}
Assume that $(W,W')$ is an exchangeable pair of real-valued random variables such that 
$$
\E (W-W'|W) = \lambda W + R
$$
for some random variable $R=R(W)$ and with $0 < \lambda < 1$. Assume that $|W-W'| \leq A$ for some $A>0$ and let $D := \frac{(W-W')^2}{2 \lambda}$. Suppose that
there exist constants $\delta_1, \delta_2$ and $\theta \geq 1$ such that
\begin{equation} \label{Cramer1}
| E ( D | W) -1 | \leq \delta_1 ( 1 + |W| ),
\end{equation}
\begin{eqnarray} \label{Cramer2}
\big| E \bigl( \frac{R}{\lambda} | W \bigr) \big| & \leq & \delta_2 ( 1 + |W| ) \,\, \text{or}  \nonumber \\
\big| E \bigl( \frac{R}{\lambda} | W \bigr)  \big| & \leq & \delta_2 ( 1 + W^2) \,\, \text{and} \,\, \delta_2 |W| \leq \alpha < 1
\end{eqnarray}
and
\begin{equation} \label{Cramer3}
| E ( D | W)| \leq \theta. 
\end{equation} 
Then
\begin{equation} \label{Cramerstate}
\frac{\PP(W >x)}{\PP(Z > x)} = 1 + {\mathcal O}_{\alpha}(1) \theta^3 (1+x^3) (A + \delta_1 + \delta_2)
\end{equation}
for $0 \leq x \leq \theta^{-1} \min( A^{-1/3}, \delta_1^{- 1/3}, \delta_2^{-1/3})$, where ${\mathcal O}_{\alpha}(1)$ denotes a quantity whose
absolute value is bounded by a universal constant which depends on $\alpha$ only under the second alternative of \eqref{Cramer2}.
\end{theorem}

In \cite{ShaoZhangZhang:2021}, Theorem 2.1, a Cram\'er-type moderate deviation result was established for nonnormal approximation, especially
for exchangeable pairs. In \cite[Theorem 2.1]{Zhang:2023}, a Cram\'er-type moderate deviation result was established for normal approximation
for unbounded exchangeable pairs. We could apply this Theorem in our situation, but prefer to apply Theorem \ref{CramerTheo}.

\bigskip

\section{Important facts on the cubic Ising mean-field model}

We appeal to the theory of large deviations to define the set of {\it canonical equilibrium macrostates}. In order to state a large deviations
principle (LDP) (for a definition see \cite{DemboZeitouni:book}, Section 1.2) for the spin per site for the Cubic mean-filed Ising model
we remind that Cram\'er's theorem \cite[Theorem 2.2.3]{DemboZeitouni:book} states that with respect to the product measure of $\frac 12 \delta_{-1} + \frac 12 \delta_1$
the sequence $\bigl( \frac{S_n}{n} \bigr)_n =(m_n)_n$ satisfies the LDP on $[-1,1]$ with speed $n$ and rate function
\begin{equation} \label{bentropy}
I(m) = \frac{1-m}{2} \log \biggl( \frac{1-m}{2} \biggr) +  \frac{1+m}{2} \log \biggl( \frac{1+m}{2} \biggr),
\end{equation}
sometimes called the (binary) relative entropy. Having this LDP, the LDP for $\bigl( \frac{S_n}{n} \bigr)_n = (m_n)_n$ with respect to $\PP_n = \PP_{n,K,J}$ is a consequence of
\cite[Theorem 2.4]{EllisHavenTurkington:2000}:

\begin{theorem}[Large deviation principle for the total magnetization] \label{LDP}
For all $(K,J) \in {\mathbb R}^2$ the following conclusion holds: with respect to  $\PP_{n,K,J}$, $\bigl( \frac{S_n}{n} \bigr)_n = (m_n)_n$ satisfies the LDP on $[-1,1]$ with speed $n$ and rate function
$$
I_{K,J}(m) = I(m) - \frac{K}{3} m^3 - \frac{J}{2} m^2 - \inf_{y \in \R} \biggl\{ I(y) - \frac{K}{3} y^3 - \frac{J}{2} y^2\biggr\},
$$
with $I$ taken from \eqref{bentropy}.
\end{theorem}
As a consequence only the points $m \in [-1,1]$ satisfying $I_{K,J}(m)=0$ do not have an exponentially small probability of being observed. These
points form the set of the so-called {\it equilibrium macrostates}, which is accordingly defined as
$$
{\mathcal M}_{K,J} = \{ m \in [-1,1]: I_{K,J}(m) = 0 \}.
$$
This set was already analyzed in \cite{CMO2024}. The critical points $m$ with 
$$
I(m) - \frac{K}{3} m^3 - \frac{J}{2} m^2  =  \inf_{y \in \R} \biggl\{ I(y) - \frac{K}{3} y^3 - \frac{J}{2} y^2 \biggr\}
$$
satisfy the equation $I'(m) - K m^2 - Jm = 0$, which is equivalent to
\begin{equation} \label{meanfieldeq}
\tanh (K m^2 + J m) = m.
\end{equation}
Proposition 2.2 in \cite{CMO2024} states that for any $K >0$, there exists a function $\gamma$ and $J = \gamma(K)$ such that ${\mathcal M}_{K,J}$
consists of a unique phase, which is a unique maximum point $m^*$ of $y \mapsto \frac{K}{3} y^3 + \frac{J}{2} y^2 -I(y)$ for $(K,J) \in (\R_+ \times \R) \setminus \{(K, \gamma(K)), K >0\}$. Moreover on the curve $\gamma$ there are {\it two distinct phases}. Interesting enough the function $\gamma$ exists as a consequence
of the intermediate value theorem, but the expression is not given explicitly. 
Moreover the limit, as $K \to 0$ of $\gamma(K)$ identifies a critical point $(K_c,J_c) = (0,1)$
where the magnetization takes the value $m_c=0$. Stated in \cite[Corollary 3.1]{CMO2024}, the set ${\mathcal M}_{K,J}$ consists of a unique pure phase
$m^*(K,J)$ given by:
$$
   m^* :=  m^*(K,J) = \left\{\begin{array}{ll}
        m_0=0, & \text{if } J < \gamma(K),\\
        m_1(K,J) & \text{if } \gamma(K) < J < J_c=1,\\
        m_2(K,J), & \text{if } J \geq J_c=1.
        \end{array} \right.
$$
It is know from \cite[Proposition 3.1]{CMO2024} that $m_1(K,J)$ and $m_2(K,J)$ are strictly positive for $K>0$. Of big interest for the model is the understanding
of critical exponents. By the law of large numbers the average value of the magnetization is $m^*$. In Theorem 2.1 in \cite{CMO2024} the law
of large numbers was proved. Outside of the critical curve, the strong law of large numbers follows as a direct consequence of the large deviations principle
Theorem \ref{LDP}. Interesting enough is the critical behaviour of $m^*$ when $(K,J) \to (K_c,J_c)=(0,1)$, proved in Proposition 2.2 in \cite{CMO2024}. 
The result will be used to observe a threshold for the speed of convergence, see Theorem \ref{threshold}.
Given the unique value $m^*(K,J)$ for all $(K,J)$ outside the curve $\gamma$, consider the affine lines
$$
J(K) = 1 + \alpha K, \,\, K>0
$$
and $m^*(K) := m^*(K, J(K))$. Then for $K \to 0^+$, the following holds
$$
m^*(K) \sim \left\{
\begin{array}{ll}
\sqrt{3 \alpha} \sqrt{K}, & \textrm{for} \,\, \alpha >0 \\
3K, & \textrm{for} \,\, \alpha =0 \\
0, & \textrm{for} \,\, \alpha < 0.
\end{array}
\right. 
$$
For $\alpha <0$ the statement is $m^*(K)=0$, if $K$ is small enough.
\bigskip

\section{(non-uniform) Berry-Esseen bounds for limit theorems for the cubic mean-field model}

\noindent
Now we consider 
\begin{equation} \label{defphi}
\phi(x) := I(x) - \frac{K}{3} x^3 - \frac{J}{2} x^2. 
\end{equation}
Hence $\phi'(x) = \frac12 \log( (1+x)/(1-x)) -Kx^2 -Jx$ and $\phi''(x)= \frac{1}{1-x^2} - 2Km - J$.
Notice that for any pure phase $m^*$ we have $\phi''(m^*) = \frac{1 - (1-(m^*)^2)(2 K m^* + J)}{1-(m^*)^2}$, and we define
\begin{equation} \label{variance}
\sigma^2 := \frac{n}{\phi''(m^*)} = \frac{n (1-(m^*)^2)}{1 - (1-(m^*)^2)(2 K m^* + J)}.
\end{equation}
and
\begin{equation} \label{magrescaled}
W := W_n :=  \frac{S_n -n m^*}{\sigma} = \frac{S_n -n m^*}{\sqrt{n / \phi''(m^*)}}.
\end{equation}

In Theorem 2.1 (part 1.) of \cite{CMO2024} the authors proved a law of large numbers and a central limit theorem: For $(K,J) \in (\R_+ \times \R) \setminus (\gamma \cup (K_c,J_c))$, the sequence $(S_n/n)_n$ converges in distribution to the Dirac measure $\delta_{m^*}$ and 
the sequence $(W_n)_n$ converges in distribution to a standard normal distributed random variable.

First of all we will prove the following Berry-Esseen bound for the rescaled magnetization of the cubic mean-field model.
The central limit theorem was proved in \cite[Theorem 2.1 (1)]{CMO2024}.

\begin{theorem}[(non-uniform) Berry-Esseen bounds for the magnetization]\label{CubicBE}
Let $Z$ be a standardnormal distributed random variable. Consider the classical cubic mean-field model given by the probability measure \eqref{BG} and consider the magnetization per particle $m_n$ \eqref{mag} and the normalized random variable $W_n$ in \eqref{magrescaled}.
Assume that  $(K,J) \in (\R_+ \times \R) \setminus (\gamma \cup (K_c,J_c))$, then there exists a constant $c=c(K,J)$ such that
$$
d_K(W_n,Z) \leq \frac{c}{\sqrt{n}},
$$
and there exists a constant $c=c(K,J,\E(W_n^2))$ such that for any $z \in \R$
$$
|\PP_n(W_n \leq z) - \PP(Z \leq z) | \leq  \frac{c}{(1+ |z|) \sqrt{n}}.
$$
\end{theorem}

\begin{remark} Choosing the critical values $(K_c=0, J_c=1)$ the model is no longer a cubic model, but the 2-spin Curie Weiss model
at the critical temperature $\beta_c = J_c =1$. It is well known that at this critical point, $n^{1/4} m_n$ converges in distribution
to a random variable $Y$, where it's distribution has the density
$
C \exp \bigl( - \frac{x^4}{12} \bigr)
$
with $C = \frac{ \sqrt[4]{3} \, \Gamma(\frac 14)}{\sqrt{2}}$, see \cite{ Ellis/Newman:1978}. This is the density $p_2$ in Section \ref{section2.2} choosing
$g_3(x)=x^3$. In \cite{ELStein} and \cite{ChatterjeeShao:2011}, the Berry-Essen bound
was proved to be $\frac{C}{\sqrt{n}}$. Hence the statement in Part 3. of Theorem 2.1 in \cite{CMO2024} follows immediately and a Berry-Esseen rate
is known as well. The non-uniform bound was proved in \cite{LLWC21}, Theorem 4.2:
$$
\big| \PP \bigl(  n^{1/4} m_n \leq z \bigr) - \PP (Y \leq z) \big| \leq \frac{c}{(1 + \frac 13 |z|^3) \, \sqrt{n}}. 
$$
In Theorem 4.2 in \cite{LLWC21} the speed was considered to be $n^{-1/4}$. But in \cite{ChatterjeeShao:2011}
as well as in \cite{ELStein} the speed $n^{-1/2}$ was already considered in the symmetric 2-spin Curie-Weiss model. Moreover, in \cite{ShaoZhang:2019},
the Berry-Esseen bound at the critical value of $J$ was considered also for non-symmetric single-spin measures and this is of order $n^{-1/2k}$
in case of $g(x)=x^{2k-1}$. But in the case of any symmetric single-spin distribution, one is able to obtain a Berry-Esseen bound of order $n^{-1/k}$,
see also Remark 3.1 in \cite{ShaoZhangZhang:2021}.
 \end{remark}

For every pair $(K,J)$ on the curve $\gamma$ there are two global maximizers of $\phi$. In this case we will prove
Berry-Esseen bounds in addition to the central limit theorem in \cite[Theorem 2.1 (2)]{CMO2024}:

\begin{theorem}[local Berry-Esseen bounds]  \label{CubicConditionBE}
If $(K,J)$ is any point on the curve $\gamma$, then for both distinct phases $m_0$ and $m_1$, we have
$$
\sup_{x \in \R} \bigg| \PP_n \bigl( W_n \leq x \big| m_n \in A_i \bigr) - \PP(Z \leq x) \bigg| \leq \frac{c}{\sqrt{n}}
$$
for a constant $c = c(K, \gamma(K))$ and $A_i$ being a neighborhood of $m_i$ whose closure does not contain the other phase, for $i=0,1$.
\end{theorem}

Next we analyze Berry-Esseen bounds in case the model-parameters $(K_n,J_n)$ converge to the critical
point $(K_c,J_c)$ with a certain speed and obtain a threshold for the speed of convergence. 
Consider $(K_n,J(K_n))$ being points on the affine line 
\begin{equation} \label{affine}
J(K_n) = 1 + \alpha K_n
\end{equation} 
for $(K_n)_n$ with $K_n>0$, $n \in \N$, and fix the {\it slope} $\alpha < 0$. It is known
from Proposition 2.3 in \cite{CMO2024} that for $K_n$ small enough that $J(K_n) < \gamma(K_n)$ (with $\gamma$ is the
curve such that for $(K_n, \gamma(K_n))$ there are two phases) and that there exists a $n_0 \in \N$ such that $m(K_n, J(K_n))=0$ for all $n \geq n_0$.
We will see, that with speed $K_n = 1/ \sqrt{n}$, although we have convergence to the critical point $(0,1)$, 
the sequence $\frac{S_n}{n^{3/4}}$ converges in distribution to a mixture of a normal distribution and the quartic distribution at the critical point. If we speed
up, the limiting distribution will be the quartic distribution, if we slow down, we obtain the Gaussian approximation for  $\frac{S_n}{n^{3/4}}$ and for all
$\frac{S_n}{n^{\varepsilon}}$ with $\frac 12 < \varepsilon < \frac34$, although we are reaching the critical point.

In Remark \ref{alphavar} we discuss the cases where the affine line has slope $\alpha =0$ ($J(K_n)=1$ for every $n \in \N$) and a positive slope $\alpha >0$.
Interesting enough, the results for $\alpha =0$ differ.

\begin{theorem}[Mixed distributions near the critical point] \label{threshold}
Consider $(K_n,J(K_n))$ with \eqref{affine} and fix $\alpha < 0$. Choose $n_0$ such that $m = m(K_n, J(K_n))=0$ for all $n \geq n_0$.

{\bf (1).} Let  $K_n = \frac{1}{\sqrt{n}}$ and
$$
W := W_n := \frac{S_n}{n^{3/4}}.
$$
Then we have
$$
d_K(W_n, Y_{\alpha, J(K_n)}) \leq \frac{c}{n^{1/4}}
$$
for $n \geq n_0$, $c= c(\alpha)$ a constant, and where $Y_{\alpha, J(K_n)}$ is a random variable with probability density
$$
c_{\alpha,J(K_n)} \exp \biggl( - (-\alpha) \frac{x^2}{2} - J(K_n)^3 \frac{x^4}{12} \biggr),
$$
with $c_{\alpha,J(K_n)}$ being the normalization factor to obtain a probability density. 
Moreover we have for any $z \in \R$
$$
\big| \PP_n(W_n \leq z) - \PP(Y_{\alpha, J(K_n)} \leq z) \big|  \leq \frac{c}{(1+|g(z)|) n^{1/4}}
$$
with $g(z) = -\alpha z + \frac 13 J(K_n)^3 z^3$.

{\bf (2).} If $K_n << \frac{1}{\sqrt{n}}$, $W_n$ converges in distribution to a random variable with density $c_{J(K_n)} \exp \biggl(- J(K_n)^3 \frac{x^4}{12} \biggr)$,
with $c_{J(K_n)}$ being the normalization factor to obtain a probability density. If $K_n = {\mathcal O} \bigl( \frac{1}{n^{3/4}} \bigr)$, the following Berry-Esseen bound
is true for $n \geq n_0$ and a constant $c$:
$$
d_K(W_n, Y_{0,J(K_n)}) \leq \frac{c}{n^{1/4}}.
$$
Moreover we have for any $z \in \R$
$$
\big| \PP_n(W_n \leq z) - \PP(Y_{0, J(K_n)} \leq z) \big|  \leq \frac{c}{(1+|g(z)|) n^{1/4}}
$$
with $g(z) = \frac 13 J(K_n)^3 z^3$.

{\bf (3).} If $K_n >> \frac{1}{\sqrt{n}}$, and let $W_n= \frac{S_n}{\sigma}$ with $\sigma = \frac{\sqrt{n}}{\sqrt{1 - J(K_n)}}$, then $W_n$ converges in distribution to a 
standardnormal distributed random variable $Z$. If we choose $K_n = n^{-2 \delta}$ with $1/6 < \delta < 1/4$ , the following Berry-Esseen bound
is true for $n \geq n_0$ and a constant $c$:
$$
d_K(W_n, Z) \leq \frac{c}{n^{1/2 - \delta}}.
$$
Moreover we have for any $z \in \R$
$$
\big| \PP_n (W_n \leq z) - \PP(Z \leq z) \big|  \leq \frac{c}{(1+|z|) n^{1/2 - \delta}}.
$$
If we choose $K_n = n^{-2 \delta}$ with $0 < \delta \leq 1/6$, the following Berry-Esseen bound
is true for $n \geq n_0$ and a constant $c$:
$$
d_K(W_n, Z) \leq \frac{c}{n^{1 - 4 \delta}}.
$$
Moreover we have for any $z \in \R$
$$
\big| \PP_n(W_n \leq z) - \PP(Z \leq z) \big|  \leq \frac{c}{(1+|z|) n^{1 - 4 \delta}}.
$$

{\bf (4).} If we choose $K_n = n^{-2 \delta}$ with $0 < \delta < 1/4$ (hence $K_n >> \frac{1}{\sqrt{n}}$, see {\bf (3).}), and let $W_n = \frac{S_n}{n^{3/4}}$, 
then $W_n$ converges in distribution to a standardnormal distributed random variable $Z$. For any $0 < \delta < \frac 18$ we obtain
$$
d_K(W_n,Z) \leq \frac{c}{n^{1/4}},
$$
where $c$ is a positive constant. The non-uniform Berry-Esseen bound with prefactor $1/(1 + |z|)$, $z \in \R$, holds as well.
\end{theorem}

\begin{remark} In the classical Curie-Weiss model $K=0$, the Berry-Esseen bound in part \textbf{(1)} is known to be of order $1/ \sqrt{n}$.
We will see, that for $K \not= 0$, the bound of order $n^{-1/4}$ cannot be improved within the method we are applying, see Remark \ref{noimprov}.
\end{remark}

\bigskip
\bigskip
The fundamental construction of the paper is to obtain the exchangeable pair $(W_n, W_n')$. For simplicity we will denote $W := W_n$ and $W' :=W_n$.
Suppose that $X = (X_i)_i$ is drawn from the Cubic mean-field measure $\PP_n$, then let $I$ be a uniformly distributed random variable on the set $\{1, \ldots, n \}$,
independent of $X$. Given $I=i$, $X'$ is constructed by taking one step in the heat-bath Glauber dynamics: coordinate $X_i$ is replaced by $X_i'$ drawn
from the conditional distribution of $X_i$ given $(X_j)_{j \not= i}$. We denote bei $W'$ the version of $W$ evaluated at $X'$, this is
$$
W' = W + \frac{X_I' - X_I}{\sigma}.
$$
For simplicity we denote $m := m^* = m^*(K,J)$ for any unique pure phase. Hence we obtain that
$$
\E ( W-W' | W) = \frac{1}{\sigma} \E(X_I - X_I' |W) = \biggl( \frac{W}{n} + \frac{m}{\sigma}  \biggr) - \frac{1}{\sigma} \frac 1n \sum_{i=1}^n \E ( X_i' | W).
$$

\begin{lemma} \label{condCubic}
We have
$$
\E ( X_i' | W) = \tanh \bigl(  J m_i(X) + K m_i^2(X) + \frac{K}{3n^2} \bigr),
$$
where
$$
m_i(x) := \frac 1n \sum_{j \not = i} x_j := \frac 1n \sum_{j=1, j \not= i}^n x_j.
$$
\end{lemma}

\begin{proof}
We compute the conditional density $g(x_1 | (X_i)_{i \geq 2})$ of $X_1=x_1$ given $(X_i)_{i \geq 2}$ under the measure $\PP_n$:
$$
g(x_1 | (X_i)_{i \geq 2}) = \frac{
\exp \biggl( \frac{J}{2n} \bigl(x_1^2 + 2 \sum_{j=2}^n x_1 X_j \bigr) + 
\frac{K}{3n^2} \bigl( x_1^3 + 3 \sum_{j=2}^n x_1^2X_j + 3 \sum_{j,k=2}^n x_1 X_j X_k \bigr) \biggr) }{ \exp\bigl( \frac{J}{2n} + \frac{K m_1(X)}{n} \bigr) \cosh(J m_1(X) + K m_1^2(X) + \frac{K}{3 n^2})}
$$
Hence we can compute $\E ( X_1' | W)$ as
\begin{eqnarray*}
\E ( X_1' | W) & = & \frac{ \frac 12 \exp \bigl( J m_1(X)  + K m_1^2(X) + \frac{K}{3 n^2} \bigr) - \frac 12 \exp \bigl( -J m_1(X)  - K m_1^2(X) - \frac{K}{3 n^2} \bigr)}{
\cosh(J m_1(X) + K m_1^2(X) + \frac{K}{3 n^2})} \\
& = & \frac{ \sinh(J m_1(X) + K m_1^2(X) + \frac{K}{3 n^2})}{\cosh(J m_1(X) + K m_1^2(X) + \frac{K}{3 n^2})} .
\end{eqnarray*}
The computation of any  $\E ( X_i' | W)$ is the same.
\end{proof}

In what follows we will apply the Taylor-expansion of the function $x \mapsto \tanh(J(m+x)+K(m+x)^2)$ up to the fourth order.

\begin{lemma} \label{TaylorCol}
Let $f(x) := \tanh \bigl( J(m+x)+K(m+x)^2 \bigr)$ for $(K,J) \in \R_+ \times \R_+$ and
$m \in \R$. The Taylor-expansion of $f$ around $0$ up to the order four reads as
$$
f(x) = \sum_{k=0}^3 c_k(J,K,m) \frac{x^k}{k!} + f^{(4)}(\eta) \frac{x^4}{4!}
$$ 
with $\eta \in (0,x)$ and
$$
c_0(J,K,m) = f(0) = \tanh( Jm + Km^2),
$$
$$
c_1(J,K,m) = f'(0) = (J + 2Km) \, \tanh'(Jm+Km^2),
$$
$$
c_2(J,K,m) = f''(0) = 2K \, \tanh'( Jm + Km^2) + (J+2Km)^2 \, \tanh''( Jm + Km^2),
$$
and
$$
c_3(J,K,m) = f'''(0) = 6K \, (J+2Km) \, \tanh''( Jm + Km^2) + (J+2Km)^3 \,  \tanh'''( Jm + Km^2).
$$
\end{lemma}
\medskip

\begin{proof}[Proof of Theorem \ref{CubicBE}]
On the way of finding an exchangeable pair identity \eqref{exid1b}, we obtain
\begin{eqnarray*}
\E ( W-W' | W) & = & \biggl( \frac{W}{n} + \frac{m}{\sigma}  \biggr) - \frac{1}{\sigma} \frac 1n \sum_{i=1}^n \tanh \bigl(  J m_i(X) + K m_i^2(X) + \frac{K}{3n^2} \bigr) \\
& = & \biggl( \frac{W}{n} + \frac{m}{\sigma}  \biggr) - \frac{1}{\sigma} \tanh \bigl(  J m(X) + K m^2(X)  \bigr) + r_1(X) 
\end{eqnarray*}
with total magnetization $m(X)$ from \eqref{mag} and  
\begin{equation} \label{r1}
r_1(X) = \frac{1}{\sigma} \frac 1n \sum_{i=1}^n \biggl(  \tanh \bigl(  J m(X) + K m^2(X)  \bigr) - \tanh \bigl(  J m_i(X) + K m_i^2(X) + \frac{K}{3n^2} \bigr) \biggr).
\end{equation}
Using $m(X) = \frac{\sigma W}{n} +m$ we will consider
$$
\frac{1}{\sigma} \tanh \biggl( J \bigl( m + \frac{\sigma W}{n} \bigr) + K \bigl( \frac{\sigma W}{n} +m \bigr)^2 \biggr).
$$
With the Taylor expansion of $x \mapsto \tanh \bigl( J (m +x) + K( m+ x)^2 \bigr)$ around $0$ of order 2, see Lemma \ref{TaylorCol}, we have
\begin{equation*}
\tanh \biggl( J m + \frac{J \sigma W}{n} + K \bigl( \frac{\sigma W}{n} +m \bigr)^2 \biggr) =  c_0(J,K,m)  + c_1(J,K,m) \frac{\sigma W}{n} 
+  r_2(\eta) \frac{\sigma^2 W^2}{2 n^2} 
\end{equation*}
with 
\begin{eqnarray} \label{r2}
r_2(\eta) = f''(\eta) & = & 2K \tanh' \bigl( J(m+\eta) + K(m+\eta)^2 \bigr) \\ & \hspace{1cm} +&  (J + 2K (m+\eta))^2  \tanh'' \bigl( J(m + \eta) + K(m +\eta)^2 \bigr) \nonumber
\end{eqnarray}
for $\eta \in (0, \sigma W/n)$.
For any equilibrium macrostate $m$ we know that $c_0(J,K,m)= \tanh (Jm + K m^2) =m$, see \eqref{meanfieldeq}. 
Moreover with the definition of $I$ and $\phi$ (\eqref{bentropy}, \eqref{defphi}) we obtain that $\tanh^{-1}(x) = Jx + Kx^2 +\phi'(x)$,
which is $x = \tanh(Jx + Kx^2 +\phi'(x))$ and differentiating leads to 
\begin{equation} \label{id}
1= \tanh' (Jx + Kx^2 +\phi'(x)) (J + 2Kx + \phi''(x)).
\end{equation} For $x=m$, using
$\phi'(m)=0$ for any equilibrium macrostate, we have 
$$
\tanh'(Jm+Km^2) = \bigl( J + 2Km + \phi''(m) \bigr)^{-1} = 1-m^2.
$$
Now it follows that
\begin{equation} \label{c1phase}
c_1(J,K,m) =(J+2Km) \,  \tanh'(Jm+ K m^2) = (J+2Km)\, (1-m^2) = 1 -(1-m^2) \phi''(m).
\end{equation}
Differentiating \eqref{id} once more and choosing $x=m$ for any equilibrium macrostate leads to
$\tanh''(Jm +Km^2) = -2m(1-m^2)$ and it follows that
\begin{eqnarray} \label{c2phase}
c_2(J,K,m) & = & 2K (1-m^2) - 2m (1-m^2) (J+2Km)^2 \nonumber \\
& = & 2m(1-m^2) \phi''(m) - (1-m^2) \phi'''(m).
\end{eqnarray}
Summarzing we obtain
 $$
 - \frac{1}{\sigma} \tanh \biggl( J m + \frac{J \sigma W}{n} + K \bigl( \frac{\sigma W}{n} +m \bigr)^2 \biggr) = -\frac{m}{\sigma} - \frac{W}{n} (J+2Km)\,(1-m^2) -
  r_2(\eta) \frac{\sigma W^2}{2 n^2}.
 $$
 Now we have checked the linear regression condition of our exchangeable pair:
 
 \begin{lemma} \label{linreg}
 In the cubic mean-field model we obtain:
 \begin{equation} \label{Steinpair}
 \E( W-W' | W) = \lambda W + R
 \end{equation}
 with
 \begin{equation} \label{lambdaCUBIC}
 \lambda = \frac 1n \bigl( 1 - (J + 2Km)\,(1-m^2) \bigr)  = \frac 1n \frac{\phi''(m)}{\phi''(m) + (J+2Km)}  = \frac 1n (1-m^2) \phi''(m)
 \end{equation}
 and
 \begin{equation} \label{reminderCUBIC}
 R = r_1(X) - r_2(\eta) \frac{\sigma W^2}{2 n^2} 
 \end{equation}
 with $r_1$ defined in \eqref{r1} and $r_2$ defined in \eqref{r2} and $\eta \in (0, \sigma W/n)$. 
 \end{lemma}
 
 \begin{remark}
 We observe that $\lambda$ in \eqref{lambdaCUBIC} satisfies $\lambda \in (0,1)$
 using the fact that $\phi''(m) > 0$ for any pure phase $m$, see \cite[Proposition 3.2 (b)]{CMO2024}. 
 Hence, we expect for any pure phase a central limit theorem for the rescaled magnetization $W_n$.
 \end{remark}
 
 \begin{remark}
 If $K_c=0$ and $J_c=1$ one can see that with $m=0$ the factor $\bigl( 1 - (1-m^2)(J + 2Km) \bigr)=0$, hence the variance of $W$ does not exist.
 This case corresponds to the 2-spin Curie-Weiss model at critical temperature $J_c=1$. It is possible to expand the $\tanh$ function 
 to the third order to be able to observe that $n^{1/4} m_n(X)$ converges in distribution and to obtain Berry-Esseen bounds, see \cite{ChatterjeeShao:2011}
 and \cite{ELStein}.
 \end{remark}
  
Now we are ready to apply Theorem \ref{CLT-1} and Theorem \ref{CLT-2}. Therefore we will calculate $\frac{1}{2 \lambda}\E( (W-W')^2 | W)$.
Applying Lemma \ref{condCubic} we have
\begin{eqnarray} \label{expDelta2}
\E( (W-W')^2 | W) & = & \frac{1}{\sigma^2} \E (X_I^2 - 2 X_I X_I' + (X_I')^2 | W) \\
&\hspace{-1.5cm} = & \hspace{-0.8cm} \frac{2}{\sigma^2} - \frac{2}{\sigma^2} \frac 1n \sum_{i=1}^n X_i \E (X_i' | W) \nonumber \\
& \hspace{-1.5cm} = &  \hspace{-0.8cm}\frac{2}{\sigma^2} - \frac{2}{\sigma^2} \frac 1n \sum_{i=1}^n X_i \tanh \biggl(J m_i(X) + K m_i^2(X) + \frac{K}{3n^2} \biggr) \nonumber \\
& \hspace{-1.5cm} = & \hspace{-0.8cm} \frac{2}{\sigma^2} - \frac{2}{\sigma^2} \frac 1n \sum_{i=1}^n X_i \tanh \biggl(J \bigl(m + \frac{\sigma W}{n} \bigr) + K \bigl(m + \frac{\sigma W}{n} \bigr)^2   \biggr) +r_3(X) \nonumber \\
& \hspace{-1.5cm} = & \hspace{-0.8cm}\frac{2}{\sigma^2} - \frac{2}{\sigma^2} \biggl(m +  \frac{\sigma W}{n} \biggr) \tanh \biggl(J \bigl(m + \frac{\sigma W}{n} \bigr) + K \bigl(m + \frac{\sigma W}{n} \bigr)^2   \biggr) +r_3(X), \nonumber
\end{eqnarray}
where
\begin{equation} \label{r3}
r_3(X) := \frac{2}{\sigma^2} \frac 1n \sum_{i=1}^n X_i  \biggl( \tanh \bigl(J m_i(X) + K m_i^2(X) + \frac{K}{3n^2} \bigr) - \tanh \bigl(J m(X) + K m^2(X)   \bigr) \biggr).
\end{equation}
With the Taylor expansion of $x \mapsto \tanh \bigl( J (m +x) + K( m+ x)^2 \bigr)$ around $0$ of order 1, see Lemma \ref{TaylorCol}, we have
\begin{eqnarray*}
\tanh \biggl(J \bigl(m + \frac{\sigma W}{n} \bigr) + K \bigl(m + \frac{\sigma W}{n} \bigr)^2   \biggr) & = & \tanh(Jm+Km^2)  + f'(\eta) \, \frac{\sigma W}{n} \\
& \hspace{-4.2 cm}= &  \hspace{-2.2 cm} m + (J + 2K(m+ \eta)) \tanh' (J(m+ \eta) + K(m+\eta)^2) \frac{\sigma W}{n} 
\end{eqnarray*}
with $\eta \in (0, (\sigma W)/n)$. Using that $\tanh'(x)$ is bounded by 1, we obtain that 
\begin{eqnarray*}
\tanh \biggl(J \bigl(m + \frac{\sigma W}{n} \bigr) + K \bigl(m + \frac{\sigma W}{n} \bigr)^2   \biggr) = \biggl( m + \frac{\sigma W}{n} r_4(\eta) \biggr)
\end{eqnarray*}
with
$$
r_4(\eta) := (J + 2K(m+ \eta)) \tanh' (J(m+ \eta) + K(m+\eta)^2).
$$
Summarizing we see that
\begin{eqnarray*}
\E( (W-W')^2 | W)& = &  \frac{2}{\sigma^2} -  \frac{2}{\sigma^2} m^2 - \frac{2 m \, W}{n\sigma}(1+ r_4(\eta))  - \frac{2 W^2}{n^2} r_4(\eta)  + r_3(X) \\
&=&  \frac{2}{\sigma^2} (1-m^2) + \frac{m \, W}{n\sigma} {\mathcal O}(1) + \frac{W^2}{n^2} {\mathcal O}(1) + r_3(X).
\end{eqnarray*}
With $ \frac{2}{\sigma^2} (1 - m^2) = 2 \lambda$ we obtain
\begin{equation} \label{firstterm}
\bigg| \frac{1}{2 \lambda} \E( (W-W')^2 | W) -1 \bigg| = \bigg| {\mathcal O}(1) \frac{m \, W}{2 \lambda n \sigma} + {\mathcal O}(1) \frac{W^2}{2 \lambda n^2} + \frac{r_3(X)}{2 \lambda} \bigg|.
\end{equation}
Note that $\lambda n \sigma$ is of order $\sqrt{n}$, hence with $\E W^2 <  \infty$ we see that $\E \big| \frac{m \, W}{2 \lambda n \sigma} \big| = {\mathcal O} (\frac{1}{\sqrt{n}})$ and $\E \big| \frac{W^2}{2 \lambda n^2} \big| = {\mathcal O} (\frac{1}{n})$.

\begin{lemma} \label{r3bound}
We have 
$$
\E  \bigg| \frac{r_3(X)}{2 \lambda} \biggr| =  {\mathcal O} (1/n).
$$
\end{lemma}

\begin{proof}
Using $|X_i| \leq 1$ and by the 1-Lipschitz property of $\tanh$,
\begin{eqnarray*}
&&  \bigg| \tanh \bigl(J m_i(X) + K m_i^2(X) + \frac{K}{3n^2} \bigr) - \tanh \bigl(J m(X) + K m^2(X)   \bigr) \bigg| \leq \\
&& \hspace{1cm} \big| J (m(X)-m_i(X)) \big| + 
 \big| K (m^2(X)-m_i^2(X)) \big| + \frac{|K|}{3n^2} \leq \frac Jn + {\mathcal O} ( 1/n )  = {\mathcal O} (1/n),
\end{eqnarray*}
we observe that $|r_3(X) | = {\mathcal O} \bigl( \frac{1}{\sigma^2 n} \bigr) =   {\mathcal O} \bigl( \frac{1}{n^2} \bigr)$. Since $\lambda$ is of order $n$, the
Lemma is proved. 
\end{proof}

\noindent
Summarizing we obtain that 
$$
\E \bigg| \frac{1}{2 \lambda} \E( (W-W')^2 | W) -1 \bigg| ={\mathcal O} \bigl( \frac{1}{\sqrt{n}} \bigr).
$$

\begin{lemma}
The reminder term $R$ in \eqref{reminderCUBIC} satisfies
$$
\frac{\E |R|}{\lambda} = {\mathcal O} \bigl( \frac{1}{\sqrt{n}} \bigr).
$$
\end{lemma}

\begin{proof}
We have seen that $R = r_1(X) -  r_2(\eta) \, \frac{\sigma W^2}{2 n^2}$ with $r_1(X)$ defined in \eqref{r1} and $r_2(\eta)$ in \eqref{r2}. The function $r_2$
is bounded, since $\tanh'(x)$ is bounded by 1 and $\tanh''(x)$ is bounded by $4/3^{3/2}$. This can be seen by $\tanh''(x) = - 2 \tanh(x)(1- \tanh^2(x))$, which
has exactly two extrema $\pm x^*$ on the real line, where $x^*$ solves the equation $\tanh^2(x) = \frac 13$. Next we bound the expectation
of the modulus of $r_1(X)$. But with the same argument as in Lemma \ref{r3bound} we obtain $\E |r_1(X)| = {\mathcal O} \bigl( \frac{1}{n^{3/2}} \bigr)$.
Since the order of $\lambda$ is $n$, the lemma is proved.
\end{proof}

We observe that $|W - W'| \leq \frac{2}{\sigma} = {\mathcal O}(1/\sqrt{n})$. Applying Theorems \ref{CLT-1} and \ref{CLT-2}, the proof of Theorem \ref{CubicBE} is given under the assumption that $\E |W| \leq 2$. Without this assumption, the proof 
for the uniform Berry-Esseen bound is done, applying Theorem \ref{ourmain}.
\end{proof}

\begin{proof}[Proof of Theorem \ref{threshold}]
Proof of part {\bf (1):}

\noindent
To be able to see the quartic part of the density, we have to apply a fourth order Taylor expansion for $f(x)= \tanh (Jx + Kx^2)$, see Lemma \ref{TaylorCol}.
In analogy to the start of the proof of Theorem \ref{CubicBE}, we obtain for every equilibrium macrostate $m$ that for $W= \frac{1}{n^{3/4}} (S_n - n m)$  
$$
\E(W-W'|W) = \biggl( \frac{W}{n} + \frac{m}{n^{3/4}} \biggr) - \frac{1}{n^{3/4}} \tanh \bigl( J m(X) + K m^2(X) \bigr) + \widetilde{r_1}(X)
$$
with $\widetilde{r_1}(X) = \frac{\sigma}{n^{3/4}} r_1(X)$ with $r_1(X)$ as in \eqref{r1}. After a careful look we obtain
\begin{eqnarray} \label{laterdisc}
\E(W-W'|W) & = & \frac{W}{n} \bigl( 1 - (1-m^2)(J+2Km) \bigr)  - \frac{W^3}{n^{3/2}} \frac{c_3(J,K,m)}{6} \nonumber \\
& - & \frac{W^2}{n^{5/4}} \frac{c_2(J,K,m)}{2} - f^{(4)}(\eta) \frac{W^4}{n^{7/4}} + \widetilde{r_1}(X).
\end{eqnarray}
Here we have used that $m(X) = W / n^{1/4} + m$.
Since we are considering the case $\alpha < 0$, we have $m=m(K_n, J(K_n))=0$ for $n$ sufficiently large.  We will apply  the general
identity \eqref{laterdisc} in a discussion in Remark \ref{alphavar}.
For $m=0$ we have $c_2(J,K,0)= 2K$ (see \eqref{c2phase}) and $c_3(J,K,0)= -2 J^3$, such that for every $(K_n, J(K_n))$
\begin{equation} \label{faster}
\E(W-W'|W) = \frac{W}{n} \bigl( 1 -J(K_n)) \bigr)  + \frac{W^3}{n^{3/2}} \frac{J(K_n)^3}{3} - \frac{K_n \, W^2}{n^{5/4}} \ - f^{(4)}(\eta) \frac{W^4}{n^{7/4}} + \widetilde{r_1}(X).
\end{equation}
Remark that $f^{(4)}(\eta)$ is depending on $(K_n, J(K_n))$ as well. With $K_n = 1/\sqrt{n}$ we have $(1-J(K_n)) = -\alpha / \sqrt{n}$, thus the
first term is $\frac{- \alpha W}{n^{3/2}}$. 
With $g(x) := -\alpha x + \frac 13 J(K_n)^3 x^3$, the regression identity is
$$
\E(W-W'|W) = \lambda g(W) + R
$$
with $\lambda = 1/n^{3/2}$ and 
\begin{equation} \label{Rbound}
R :=   - \frac{K_n \, W^2}{n^{5/4}} \ - f^{(4)}(\eta) \frac{W^4}{n^{7/4}} + \widetilde{r_1}(X).
\end{equation}
Now we will apply Theorem \ref{CLT-3}. With the definition of $W$ we observe that
$$
\E( (W-W')^2 | W)  =  \frac{2}{n^{3/2}} - \frac{2}{n^{3/2}} \biggl(m +  \frac{W}{n^{1/4}} \biggr) 
\tanh \biggl(J \bigl(m + \frac{W}{n^{1/4}} \bigr) + K \bigl(m + \frac{W}{n^{1/4}} \bigr)^2   \biggr) + \widetilde{r_3}(X)
$$
with $\widetilde{r_3}(X) = \frac{\sigma^2}{n^{3/2}} r_3(X)$, where $r_3(X)$ is defined in \eqref{r3}.

Analogously to the proof of Theorem \ref{CubicBE} we obtain by first order Taylor expansion around 0, that
\begin{equation} \label{expDelta2b}
\E( (W-W')^2 | W) = \frac{2}{n^{3/2}} -  \frac{2}{n^{3/2}} m^2 + \frac{m \, W}{n^{7/4}} {\mathcal O}(1) + \frac{W^2}{n^2} {\mathcal O}(1) + \widetilde{r_3}(X).
\end{equation}
Now with $m=0$ and $\frac{2}{n^{3/2}} = 2 \lambda$, we get
$$
\bigg| \frac{\E( (W-W')^2 | W)}{2 \lambda} -1 \bigg| = \bigg| \frac{W^2}{2 \lambda n^2} {\mathcal O}(1) + \frac{\widetilde{r_3}(X)}{2 \lambda} \bigg|.
$$
Since $\E | \frac{W^2}{2 \lambda n^2} | = {\mathcal O} \bigl( 1 / \sqrt{n} \bigr)$ and with the help of Lemma \ref{r3bound}  $\E \big| \widetilde{r_3}(X) \big| = 
{\mathcal O} \bigl( 1 / n^{5/2} \bigr)$, we have
$$
\E \bigg| \frac{\E( (W-W')^2 | W)}{2 \lambda} -1 \bigg|  = {\mathcal O} \bigl( 1 / \sqrt{n} \bigr).
$$
The second term in the bound \label{SZ19mainnonnormal} of Theorem \ref{CLT-3} is $\frac{\E |R|}{c_k \lambda}$. Here $R$ is given by \eqref{Rbound}.
With $K_n = 1/\sqrt{n}$ and $\lambda = n^{-3/2}$ we get $\E \big| \frac{K_n W^2}{n^{5/4} \lambda} \big| = {\mathcal O} (1/n^{1/4})$. Next 
we have $\frac{\E |\widetilde{r_1}(X)|}{2 \lambda} = {\mathcal O}(1/n^{1/4})$. Finally we observe that $f^{(4)}(\eta)$ is bounded (using the boundedness
of $\eta$), and therefore the last term
can be bounded by $\frac{1}{\lambda} \E \big| f^{(4)}(\eta) W^4 / n^{7/4} \big| = {\mathcal O}( 1/n^{1/4})$. Finally $|W-W'| \leq \frac{2}{n^{3/4}} =:A$.
Thus with Theorem \ref{CLT-3} and Theorem \ref{CLT-4}, part {\bf (1)} of the Theorem is proved.
\medskip

\noindent
Proof of part {\bf (2):}

\noindent
Now we assume that $K_n$ converges faster to zero than $1/ \sqrt{n}$. Now the first summand on the right hand side
of the equality \eqref{faster} is
$$
\frac{W}{n} \bigl( 1 -J(K_n)) = - \frac{\alpha K_n W}{n} 
$$
and $\E \bigg| \frac{ - \alpha K_n W}{\lambda n} \bigg| = {\mathcal O} (K_n \sqrt{n}) = o(1)$. Hence this term will be a new member of the
reminder term and this observation leads to the new regression identity
$$
\E(W-W'|W) = \lambda \widetilde{g}(W) + R - \frac{\alpha K_n W}{n} 
$$
with $\widetilde{g}(x) = \frac 13 J(K_n)^3 x^3$ and $R$ defined in \eqref{Rbound}.  Now we obtain convergence in distribution for any $K_n << 1/ \sqrt{n}$ applying Theorem \ref{CLT-3}.
Moreover we obtain the Berry-Esseen bound of order $n^{-1/4}$ if we choose $K_n= n^{-3/4}$. If we let $K_n$ converge to zero even faster, this
does not influence the order of  the bound of $R$. Hence the rate of convergence $n^{-1/4}$ seem to be optimal within the
method we have applied.
\medskip

\noindent
Proof of part {\bf (3):}

\noindent
Now we change to the random variable $W := \frac{S_n}{\sigma}$ with $\sigma^2 = \frac{n}{1-J(K_n)}$. We went back to the proof of Theorem \ref{CubicBE}. 
We apply Taylor's expansion of third order with the help of Lemma \ref{TaylorCol}. Now let $\lambda = \frac{1 - J(K_n)}{n} = - \frac{\alpha K_n}{n}$ with $\alpha <0$
and $m=0$, we obtain the following representation of $\E(W-W'|W)$: 
\begin{equation*}
\E(W-W'|W) = \frac{W}{n} (1- J(K_n)) - \frac{K_n \sigma W^2}{n^2}  - \frac 16 \frac{\sigma^2 W^3}{n^3} r_5(\eta) + r_1(X),
\end{equation*}
where $r_5(\eta)$ is given by
\begin{eqnarray*}
r_5(\eta) & = & 6 K_n (J(K_n) + 2 K_n \eta)  \tanh''(J(K_n) \eta + K_n \eta^2) \\
& \hspace{1cm} +& \hspace{0cm} (J(K_n) +2K_n \eta)^3 \tanh'''(J(K_n) \eta + K_n \eta^2)
\end{eqnarray*}
with $\eta \in (0, \sigma W/n)$. We have $|W-W'| \leq \frac{2}{\sigma} = \frac{2 \sqrt{1-J(K_n)}}{\sqrt{n}} =\frac{2 \sqrt{- \alpha}  \sqrt{K_n}}{\sqrt{n}} =:A$,
which is converging to zero by our assumption on $K_n$. Since $m=0$, it follows from \eqref{firstterm} that 
$$
\bigg| \frac{1}{2 \lambda} \E( (W-W')^2 | W) -1 \bigg| = \bigg| {\mathcal O}(1) \frac{W^2}{2 \lambda n^2} + \biggl( \frac{r_3(X)}{2 \lambda} \biggr) \bigg|.
$$
We have
$$
\E \bigg| \frac{r_3(X)}{\lambda} \bigg| = {\mathcal O} ((\sigma^2 n \lambda)^{-1}) =  {\mathcal O} (n^{-1}) \quad \text{and} \quad
\E \bigg| \frac{W^2}{\lambda n^2} \bigg| = {\mathcal O} \bigl( \frac{1}{n K_n} \bigr),
$$
and the last bound is converging to zero by our assumption on $K_n$. 
Finally, since $r_5(\eta)$ is bounded by a constant, we consider the following three bounds:
$$
\E \bigg| \frac{r_1(X)}{\lambda} \bigg| = {\mathcal O} ((\sigma n \lambda)^{-1}) =  {\mathcal O} \bigl( \frac{1}{\sqrt{ n K_n}} \bigr), 
$$
$$
\E \bigg| \frac{K_n \sigma W^2}{\lambda n^2} \bigg| = {\mathcal O} \bigl( \frac{1}{\sqrt{n K_n}} \bigr) \quad \text{and} \quad
\E \bigg| \frac{\sigma^2 W^3}{\lambda n^3} \bigg| = {\mathcal O} \bigl( \frac{1}{n K_n^2} \bigr).
$$
With $K_n >> 1/ \sqrt{n}$ all terms converge to zero, hence convergence to the standardnormal distribution is proved.

Since we assume that $K_n$ converges to zero slower than $n^{-1/2}$, we choose without loss of generality $K_n = n^{-2 \delta}$ for $0 < \delta < 1/4$.
It follows that $\lambda = \frac{1}{n} (1-J(K_n)) = - \frac{\alpha K_n}{n}  = - \alpha \frac{1}{n^{1 + 2 \delta}}$. With $\sigma^2 = \lambda^{-1}$ we obtain the
following bounds:
$$
\frac{|W-W'|}{\sigma} = \frac{|X_I - X_I'|}{\sigma} \leq (2 \sqrt{-\alpha}) \frac{1}{n^{1/2 +\delta }} =: A_n,
$$
and
$$
 {\mathcal O} \bigl( \frac{1}{n K_n} \bigr) =  {\mathcal O} \bigl( \frac{1}{n^{1 - 2 \delta}} \bigr), \,\, {\mathcal O} \bigl( \frac{1}{\sqrt{n K_n}} \bigr) = {\mathcal O} 
 \bigl( \frac{1}{n^{1/2 - \delta}} \bigr), \,\, {\mathcal O} \bigl( \frac{1}{n K_n^2} \bigr) =  {\mathcal O} \bigl( \frac{1}{n^{1 - 4 \delta}} \bigr).
 $$
 Summarizing the rate of convergence is $\frac{1}{n^{1 - 4 \delta}}$ or $\frac{1}{n^{1/2 - \delta}}$ for $0 < \delta < \frac 14$. Hence the
 rate of convergence is $\frac{1}{n^{1/2 - \delta}}$
 for $ 0 <  \delta \leq \frac 16$ and the rate of convergence is  $\frac{1}{n^{1 - 4 \delta}}$ for $\frac 16 <  \delta < \frac 14$. In the first case it
 is slower than $1/\sqrt{n}$, in the second case it is slower than $1/n^{1/3}$.
\medskip
 
\noindent
Proof of part {\bf (4):} With $W_n = \frac{S_n}{n^{3/4}}$ we jump back to \eqref{faster}. With $K_n = n^{-2 \delta}$, we obtain the regression identity
$$
\E(W-W'|W) = \lambda W + \widetilde{R}
$$
with $\lambda = - \alpha \frac{1}{n^{1 + 2 \delta}}$ and $\widetilde{R} := R - \frac{W^3}{n^{3/2}} \frac{J(K_n)^3}{3}$ with $R$ defined in \eqref{Rbound}.
We have that $|W-W'| \leq{2}{n^{3/4}}$, and $\E \big|\frac{R}{\lambda} \big|= {\mathcal O} \bigl( n^{2 \delta -1/2} \bigr) $. This is because we have
$$
\E \bigg| \frac{K_n W^2}{\lambda n^{5/4}} \bigg| = {\mathcal O}  \bigl( n^{-1/4} \bigr), \quad \E \bigg| \frac{W^3}{\lambda n^{3/2}} \bigg| = {\mathcal O}   \bigl(n^{-1/2 + 2 \delta} \bigr),
$$
and
$$
\E \bigg| \frac{f^{(4)}(\eta) W^4}{\lambda n^{7/4}} \bigg| = {\mathcal O}   \bigl(n^{-3/4+2 \delta} \bigr), \quad \E \bigg| \frac{\widetilde{r_1}(X)}{\lambda} \bigg| = {\mathcal O}   \bigl(n^{-3/4 + 2 \delta} \bigr).
$$
With
$$
\E \bigg| \frac{W^2}{2 \lambda n^{2}} \bigg| = {\mathcal O}  \bigl( n^{-1 + 2 \delta} \bigr), \quad \E \bigg| \frac{\widetilde{r_3}(X)}{2 \lambda} \bigg| = {\mathcal O}   \bigl(n^{-3/2 + 2 \delta} \bigr),
$$
we obtain $\E \big| \frac{\E((W-W')^2|W)}{2 \lambda} -1 \big| =  {\mathcal O}  \bigl( n^{-1 + 2 \delta} \bigr)$. Hence the convergence in distribution is proved and with
$0 < \delta < 1/8$ we obtain the (non-uniform) Berry-Esseen bound.
\end{proof}

\begin{remark} \label{noimprov}
In the proof of part \textbf{(1)}, we have seen that every term of the remainder $R$ in \eqref{Rbound} is of order $n^{-1/4}$. The other terms
are of order $1/ \sqrt{n}$ or even smaller.
We saw that $\frac{\E |\widetilde{r_1}(X)|}{2 \lambda} = {\mathcal O}(1/n^{1/4})$, using the Lipschitz property of $\tanh(x)$. Applying for any $x,y \in \R$
$$
| \tanh(x) - \tanh(y) - (x-y) (\cosh y)^{-2} | \leq \frac{2(x-y)^2}{3^{3/2}}
$$
(see \cite[page 480]{ChatterjeeShao:2011}), one can do some calculations to obtain that  $\frac{\E |\widetilde{r_1}(X)|}{2 \lambda} = {\mathcal O}(1/n^{1/2})$.
Moreover, the fourth derivative $f^{(4)}(0)$ has a prefactor $K_n$, hence if we apply Taylor expansion in \eqref{laterdisc} up to the fifth order, we 
can obtain that $\frac{1}{\lambda} \E \big| f^{(4)}(0) W^4 / n^{7/4} \big| = {\mathcal O}( 1/n^{1/2})$ and that   
$\frac{1}{\lambda} \E \big| f^{(5)}(\eta) W^5 / n^{2} \big| = {\mathcal O}( 1/n^{1/2})$. Nevertheless, the fact that $c_2(J(K_n),K_n,0) = 2 K_n \not=0$
determines the Berry-Esseen bound of order $n^{-1/4}$.  The exchangeable pair approach is limited at that point.
 \end{remark}

\begin{remark} \label{alphavar}
Let us briefly discuss the statements of Theorem \ref{threshold} in the case $\alpha \geq 0$. This slope of $J(K_n) = 1 + \alpha K_n$ implies $J(K_n) \geq J_c =1$
and therefore $m^*(K_n) = m_2(J(K_n),K_n)$ is the only positive (!) solution of $\tanh(Km^2+Jm)=m$. 

If $\alpha=0$, Proposition 2.3 in \cite{CMO2024} gives $m^*(K_n) \sim 3 K_n$
by Taylor expansion. With the notion $m(n):=m^*(K_n) \sim 3 K_n$, we consider the expansion in \eqref{laterdisc}. 
We observe that $1-(1-m(n)^2)(1+ 2 K_n m(n)) \sim 3 K_n^2$, $c_3(1,K_n,m(n)) = -2 +o(1)$ and $c_2(1,K_n,m(n)) \sim 2 K_n$. Hence we approximately have
$$
\E(W-W'|W) \sim \frac{W}{n} 3 K_n^2 + \frac{W^3}{3 n^{3/2}}  - \frac{W^2}{n^{5/4}} K_n - f^{(4)}(\eta) \frac{W^4}{n^{7/4}} + \widetilde{r_1}(X).
$$
Now the choice $K_n = 1/n^{1/4}$ leads to the same prefactor $1/n^{3/2}$ for the first 3 terms, but we cannot handle the function $g(x) = 3x + x^3/3 - x^2$
within the approach of exchangeable pairs. $g$ does not satisfy the assumptions in \cite{ELStein}, \cite{ChatterjeeShao:2011} or
\cite{ShaoZhang:2019}. The choice $K_n = 1/n^{1/2}$ implies the quartic limiting behaviour, the Gaussian part does not survive. Every sequence $K_n$ 
which converges to zero slower than $n^{-1/2}$ can be presented as $K_n = n^{-1/2+ \varepsilon}$ with $0 < \varepsilon < 1/2$. Now the prefactor
of $W$ is of order $n^{-2 + 2 \varepsilon}$ and the choice $\lambda = n^{-2 + 2 \varepsilon}$ with $1/4 < \varepsilon < 1/2$ leads to a normal approximation
of $W$ since $2 -2 \varepsilon < 3/2$ and $2 -2 \varepsilon < 7/4 - \varepsilon$. Here the quartic limit does not survive. Moreover, like in part {\bf (3).}
of Theorem \ref{threshold},  normal approximation for $W_n = S_n / \sigma$ follows for $K_n = n^{-1/2+ \varepsilon}$. 

If $\alpha >0$, Proposition 2.3 in \cite{CMO2024} gives $m^*(K_n) \sim \sqrt{3 \alpha K_n}$. Again we consider the expansion in \eqref{laterdisc}.
Now we have $1-(1-m(n)^2)(1+ 2 K_n m(n)) \sim 2 \alpha K_n$, $c_3(1,K_n,m(n)) = -2 +o(1)$ and $c_2(1,K_n,m(n)) \sim 2 K_n$. Hence approximately we
get
$$
\E(W-W'|W) \sim \frac{W}{n} 2 \alpha K_n + \frac{W^3}{3 n^{3/2}}  - \frac{W^2}{n^{5/4}} K_n - f^{(4)}(\eta) \frac{W^4}{n^{7/4}} + \widetilde{r_1}(X).
$$
Now the choice $K_n=1/\sqrt{n}$ lead to the same results as in part {\bf(1)} of Theorem \ref{threshold}, in exactly the same manner.
Moreover the statements of part {\bf(2)} - {\bf(4)} of Theorem \ref{threshold} should be the same as well. The details need some technical precisions. We omit the details. 
\end{remark}
\medskip

\noindent
For the proof of Theorem \ref{CubicConditionBE}, it will be helpful to use an alternative representation of $\E(W-W'|W)$  which will be developed
in the proof of the Cram\'er-type moderate deviations, see Theorem \ref{CubicCramer} in Section 6.  The proof of Theorem \ref{CubicConditionBE} will therefore be positioned in Section 6.

\bigskip
\section{Concentration inequality for the magnetization in the cubic mean-field model}

In the cubic mean-field model we consider $F(X,X') := \sum_{i=1}^n (X_i - X_i') = X_I - X_I'$, where $X$ are the spins under $\PP_n$ and
$X'$ is the constructed sequence due to the heat-bath Glauber dynamics.  We know that
$$
f(X) = \E (F(X, X')|X) = m(X) - \frac 1n \sum_{i=1}^n \tanh \bigl( J m_i(X) + K m_i^2(X) + \frac{K}{3n^2} \bigr).
$$
We know that $|F(X,X')|\leq 2$ because $X$ and $X'$ differ at only one coordinate. Moreover, since $\tanh(x)$ is 1-Lipschitz, we obtain
$$
|f(X) - f(X')| \leq |m(X) -m(X')| + J |m_i(X) -m_i(X')| + K |m_i^2(X) -m_i^2(X')|.
$$
We observe that $ |m(X) -m(X')| \leq \frac 2n$. The difference $ |m_i^2(X) -m_i^2(X')|$ is only nonzero if in $\frac{1}{n^2} \sum_{i,j=1}^n (X_i X_j - X_i' X_j')$
one index is $I$ and the other not. Hence the difference is equal to $\frac 1n (X_I -X_I') m(X_I)$ and appears twice, so it can be bounded by $\frac 4n$.
Summarizing we obtain
$$
 |f(X) - f(X')| \leq (J+1) \frac 2n + K \frac 4n = \frac{2(J+1 +2K)}{n}.
 $$
 It follows that $\Delta(X) \leq |E (|f(X) - f(X')| | X) \leq \frac{2(J+1 +2K)}{n} =: C$.
 Theorem \ref{Chat} implies
 $$
 \PP_n \biggl( |f(X)| \geq \frac{t}{\sqrt{n}} \biggr) \leq 2 \exp \biggl( - \frac{t^2}{4(1+J+2K)} \biggr).
 $$
 Once more using the 1-Lipschitz property of $\tanh(x)$ we obtain that
 $$
 \bigg| \frac 1n \sum_{i=1}^n \tanh \bigl( J m_i(X) + K m_i^2(X) + \frac{K}{3 n^2} \bigr) - \tanh( J m(X) + K m^2(X)) \bigg| \leq \frac{J + 4K}{n}.
 $$
 We have proved:
 
\begin{theorem}[Concentration inequality for the total magnetization]
The magnetization per particle $m(X)$ in the cubic mean-field model satisfies the following concentration inequality for all $J>0$, $K \geq 0$, $n \geq1$ and
$t \geq 0$
$$
\PP_n \biggl( |m(X) - \tanh (J m(X) + K m^2(X))| \geq \frac{J+4K}{n} + \frac{t}{\sqrt{n}} \biggr) \leq 2 \exp  \biggl( - \frac{t^2}{4(1+J+2K)} \biggr).
 $$
 \end{theorem}

\begin{remark}
The critical case $(K,J)=(0,1)$ is the 2-spin Curie-Weiss model in the critical temperature $J=J_c=1$. Here in \cite{ChatterjeeDey:2010}
it was proved that
$$
P \bigl( n^{1/4} |m(X)| \geq t \bigr) \leq 2 \exp \bigl( - c \, t^4 \bigr)
$$
for any $n\geq 1$ and $t \geq 0$. $c>0$ is an absolute constant.
\end{remark}

\bigskip

\section{Cram\'er-type moderate deviations for the cubic mean-field model}

We will prove the following deviations result:

\begin{theorem}[Cram\'er-type moderate deviations for the total magnetization]\label{CubicCramer}
Let $Z$ be a standardnormal distributed random variable. Consider the classical cubic mean-field model given by the probability measure \eqref{BG} and consider the magnetization per particle $m_n$ and the normalized random variable $W_n$ in \eqref{magrescaled}.
Assume that  $(K,J) \in (\R_+ \times \R) \setminus (\gamma \cup (K_c,J_c))$, then there exists a constant $c>0$ such that
$$
\frac{\PP_n( W_n > x)}{\PP (Z > x)} = 1 +  {\mathcal O}(1)  (1+x^3)  \bigl( \frac{1}{\sqrt{n}} \bigr)
$$
for all $x \in [0, c n^{1/6}]$. 
If $(K,J)$ is any point on the curve $\gamma$, then for both distinct phases $m_0$ and $m_1$, then there exists constants $c_i>0$, $i=0,1$, such that
\begin{equation} \label{Cramercond}
\frac{\PP_n \bigl( W_n > x \big| m_n \in A_i \bigr)}{\PP(Z > x)}  =   1 +  {\mathcal O}(1)  (1+x^3)  \bigl( \frac{1}{\sqrt{n}} \bigr)
\end{equation}
for all $x \in [0, c_i n^{1/6}]$, and $A_i$ being a neighborhood of $m_i$ whose closure does not contain the other phase, for $i=0,1$.
\end{theorem}

\begin{remark}
For the classical Curie-Weiss model ($K=0$), in \cite{Zhang:2023} the convergence range was proved to be $x \in [0, c n^{1/2}]$ instead of
$x \in [0, c n^{1/6}]$. 
\end{remark}
\begin{remark}
At the critical point $(K,J)=(0,1)$, we are considering the classical Curie-Weiss model at the critical temperature $J_c=1$. Corollary 3.1 in \cite{ShaoZhangZhang:2021}
states that
$$
\frac{\PP_{n,0,1}( S_n/n^{3/4} > x)}{\PP (Y > x)} = 1 +  {\mathcal O}(1)  (1+x^6)  \bigl( \frac{1}{\sqrt{n}} \bigr)
$$
for $x \in [0, c n^{1/12}]$, where $Y$ has the density $C \exp \bigl( -\frac{x^4}{12} \bigr)$.
\end{remark}

\begin{proof}[Proof of Theorem \ref{CubicCramer}]
We will adapt the truncation argument presented in \cite{CFS13}. Let $\widetilde{W_n}$ have the conditional distribution of $W_n$ given $|W_n| \leq d \sqrt{n}$, where
$d$ is to be determined. If we can prove that 
$$
\frac{\PP_n( \widetilde{W_n} > x)}{\PP (Z > x)} = 1 +  {\mathcal O}(1)  (1+x^3)  \bigl( \frac{1}{\sqrt{n}} \bigr)
$$
for all $x \in [0, c n^{1/6}]$, we observe the result from the fact that 
$$
\PP_n ( |W_n| > K \sqrt{n} ) \leq e^{-n C(K)}
$$
for any positive $K$ and a positive constant $C(K)$. The fact is an immediate consequence of the LDP in Theorem \ref{LDP}, since we have a unique phase
$m$ which is the only value such that the rate function $I_{K,J}(m)=0$. Consider the inequality $\PP_n(W_n >x)\leq \PP_n (\widetilde{W_n} >x) + \PP_n (\delta_2 |W_n| > 1/2)$
with $\delta_2 = {\mathcal O}(1/ \sqrt{n})$. 

We still denote $\widetilde{W_n}$ by $W$. Now we consider an interval $[a,b]$ such that mean value $m(X)$  of the spins under $\PP_n$ takes it possible
values in $[a,b]$, the only phase $m$ is in $[a,b]$ and where $[a,b]$ is choosen such that $|W_n| \leq d \sqrt{n}$. Now we consider a slightly different point of view  when calculating $\E(W-W'|W)$,
where again $W-W' =\frac{X_I-X_I'}{\sigma}$ with $\sigma$ defined in \eqref{variance}. We have
\begin{eqnarray*}
\E(W-W'|W) & = & \frac{1}{\sigma} \E(X_I - X_I' |W) = \frac{2}{\sigma} \E \bigl( 1_{\{X_I=1, X_I' = -1)\}} - 1_{\{X_I=-1, X_I' = 1)\}} |W \bigr) \\
& = & \frac{2}{\sigma}  \frac 1n \sum_{i=1}^n \biggl( 1_{\{X_i=1 \}} \PP_n(X_i' = -1|W) - 1_{\{X_i =-1\}} \PP_n( X_i' =1  |W ) \biggr). 
\end{eqnarray*}
The number of cases $1_{\{X_i=1 \}}$ is equal to $\frac{S_n}{2} + \frac n2$, and the number of cases $1_{\{X_i=-1 \}}$ is   $- \frac{S_n}{2} + \frac n2$.
Remark that
 $\frac{S_n}{2} + \frac n2= \frac 12 (\sigma W +nm +n)$ and $-\frac{S_n}{2} + \frac n2 = \frac 12 (n- \sigma W -nm)$. For $X_i=1$
 we have 
 $$
 \sum_{i=1}^n  \PP_n(X_i' = -1|W) = A(m(X)) := \frac{e^{- J m(X) - K m(X)^2}}{2 \cosh( J m(X) + K m(X)^2)} 
 $$
 and
 $$
  \sum_{i=1}^n  \PP_n(X_i' = 1|W) = B(m(X)) := \frac{e^{ J m(X) + K m(X)^2}}{2 \cosh( J m(X) + K m(X)^2)}.
 $$
 If $X_i=1$ and $X_i'=-1$ we have to ensure that $S_n - 2 \geq an$, define $A_{a} := \{ S_n -2 \geq an\}$. 
  If $X_i=-1$ and $X_i'=1$ we have to ensure that $S_n + 2 \leq bn$, define $A_{b} := \{ S_n +2 \leq bn\}$.
 Summarizing we obtain 
 \begin{eqnarray*}
 \E(W-W'|W) & \hspace{-0.3cm}=& \frac{2}{\sigma n}  \biggl(\frac{A(m(X))}{2} (\sigma W +nm +n) 1_{A_a}
 -  \frac{B(m(X))}{2} (n- \sigma W -nm)1_{A_b} \biggr) \\
 & \hspace{-2.3cm}= & \hspace{-1.2cm} \bigl( A(m(X)) + B(m(X)) \bigr) \biggl( \frac{W}{n} + \frac{m}{\sigma} \biggr) + \frac{1}{\sigma} \bigl(A(m(X)) - B(m(X)) \bigr) + R_1(X) \\
 & \hspace{-2.3cm} = & \hspace{-1.2cm} \biggl( \frac{W}{n} + \frac{m}{\sigma} \biggr) - \frac{1}{\sigma}  \tanh \bigl(J m(X) + K m(X)^2 \bigr) + R_1(X)
 \end{eqnarray*}
 with
 $$
 R_1(X) = - A(m(X)) \frac{\sigma W +nm +n}{n \sigma} 1_{A_a^c}
+  B(m(X)) \frac{n- \sigma W -nm}{n \sigma} 1_{A_b^c}.
$$
Now we proceed as in the proof of Theorem \ref{CubicBE} and obtain by Taylor expansion 
\begin{equation} \label{new1}
\E(W-W'|W) = \lambda W + R
\end{equation}
with 
$$ 
R = R_1(X) - r_2(\eta) \frac{\sigma W^2}{2 n^2},
$$
where $\eta \in (0, \sigma W/n)$ and $r_2(\eta)$ is given as in \eqref{r2}. 
Next we consider
$$
\E((W-W')^2|W) = \frac{4}{\sigma^2} \E \bigl( 1_{\{X_I=1, X_I' = -1)\}} - 1_{\{X_I=-1, X_I' = 1)\}} |W \bigr).
$$
Similarly to \eqref{expDelta2}, we have
\begin{equation} \label{new2}
\E((W-W')^2|W) = \frac{2}{\sigma^2}(1-m^2) + \frac{m W}{n \sigma} {\mathcal O}(1) + \frac{W^2}{n^2} {\mathcal O}(1)  + {\mathcal O} \biggl( \frac{1_{A_a^c \cup A_b^c}}{\sigma^2} \biggr)
\end{equation}
and with $\frac{2}{\sigma^2}(1-m^2) = 2 \lambda$ we obtain the identity
$$
\bigg| \frac{1}{2 \lambda} \E( (W-W')^2 | W) -1 \bigg| = \bigg| {\mathcal O}(1) \frac{m \, W}{2 \lambda n \sigma} + {\mathcal O}(1) \frac{W^2}{2 \lambda n^2} + 
{\mathcal O} \biggl( \frac{1_{A_a^c \cup A_b^c}}{ 2 \lambda \sigma^2} \biggr)
 \bigg|.
$$
Since $\big| \frac{W}{n} \big| = \big| \frac{m(X)-m}{\sigma} \big| = {\mathcal O}( 1/ \sigma)$ we have that ${\mathcal O}(1) \big| \frac{W^2}{2 \lambda n^2} \big|= {\mathcal O}(1) \big| \frac{W}{2 \lambda n \sigma} \big|$. We know already that $\E \big| \frac{W}{2 \lambda n \sigma} \big|= {\mathcal O}(1/ \sqrt{n})$. Moreover if $|W_n| \leq d \sqrt{n}$, then $|m(X) -m| \leq \frac{d}{\phi''(m)}$, see \eqref{variance}. Since $m \in (a,b)$ we can find a constant $\delta>0$ such that $A_a^c \cup A_b^c \subset \{ |m(X)-m| >\delta\}$. Now we choose $d$ small enough
such that $\{ |m(X)-m| > \delta\}$ is empty. Summarizing we obtain
$$
\bigg| \frac{1}{2 \lambda} \E( (W-W')^2 | W) -1 \bigg|  =  {\mathcal O}(1/ \sqrt{n}) (1 +|W|),
$$
hence condition \eqref{Cramer1} of Theorem \ref{CramerTheo} is satisfied with $\delta_1 = {\mathcal O}(1/ \sqrt{n})$. With $|W| = {\mathcal O}(\sqrt{n})$,
condition\eqref{Cramer3} of Theorem \ref{CramerTheo} is satisfied with $\theta = {\mathcal O}(1)$. Finally we have
$$
\big| \E\bigl(\frac{R}{\lambda}|W\bigr) \big| \leq {\mathcal O}(1/ \sqrt{n}) (1 + W^2)
$$
and with $\delta_2 = {\mathcal O}(1/ \sqrt{n})$ we can choose the $d$ in condition $|W_n| \leq d \sqrt{n}$ such that $\delta_2 |W| \leq \frac 12$. Hence
the second alternative of \eqref{Cramer2} in Theorem \ref{CramerTheo} is satisfied with $\alpha = \frac 12$. From Theorem \ref{CramerTheo}, we have proved
the result.

For $(K,J)$ being any point on the curve $\gamma$, we consider $\PP_n(W_n | m_n \in A_i)$, where $A_i$ is a neighborhood of $m_i$ whose closure does not contain the other phase. But know, intersect $A_i$ with the event $\{ |W_n| \leq d \sqrt{n} \}$, where $d$ is to be determined. Again an interval $[a,b]$ is choosen such that
the mean value of the spins under this conditional distribution of $W_n$ takes it possible values in $[a,b]$, meaning that only one of the two phases is in the interval and
the interval is choosen such that $|W_n| \leq d \sqrt{n}$. Now all the steps of the proof for values $(K, \gamma(K))$ are exactly the same. This completes the proof of
\eqref{Cramercond}.
\end{proof}

\begin{proof}[Proof of Theorem \ref{CubicConditionBE}]
To obtain the conditional probability $\PP_n (W_n \leq x | m_n \in A_i)$ we proceed as in the Proof of Theorem \ref{CramerTheo}. We choose an interval $[a,b]$
such that the mean value of the spins under this conditional distribution of $W_n$ takes it possible values in $[a,b]$, meaning that only one of the two phases
is in the interval. Now we consider the representation \eqref{new1} of $\E(W-W'|W)$ and \eqref{new2} of $\E((W-W')^2|W)$ and proceed as in the proof of Theorem \ref{CubicBE}. This completes the proof of the Theorem.
\end{proof}

The Cram\'er type moderate deviation results in Theorem \ref{CramerTheo} imply so called moderate deviations principles. 
Note that for $K=0$, moderate deviation principles for the magnetization in the Curie-Weiss model were proved in \cite{Eichelsbacher/Loewe:2003}.

\begin{theorem}[Moderate deviation principle for the total magnetization] \label{MDP}
Consider any sequence of real numbers $(a_n)_n$ with $1 << a_n << n^{1/6}$. Consider the classical cubic mean-field model given by the probability measure \eqref{BG} and consider the magnetization per particle $m_n$ and the normalized random variable $W_n$ in \eqref{magrescaled}. Assume that  $(K,J) \in (\R_+ \times \R) \setminus (\gamma \cup (K_c,J_c))$. Then $(\frac{1}{a_n}W_n)_n$ satisfies the moderate deviations
principle, which is a large deviations principle with speed $a_n^2$ and rate function $\frac{x^2}{2}$ in the sense of the definition \cite[Section 1.2]{DemboZeitouni:book}. In particular
we have for any $x \in \R$
$$
\lim_{n \to \infty} \frac{1}{a_n^2} \PP_n \bigl( W_n/a_n > x ) = -\frac{x^2}{2}.
$$
If $(K,J)$ is any point on the curve $\gamma$, then for both distinct phases $m_0$ and $m_1$ let $A_i$ being a neighborhood of $m_i$ whose closure does not contain the other phase, for $i=0,1$. Then we obtain a conditional moderate deviations principle, including
$$
\lim_{n \to \infty} \frac{1}{a_n^2} \PP_n \bigl( W_n/a_n > x  | m_n \in A_i) = -\frac{x^2}{2}, i=0,1,
$$
for any $x \in \R$.
\end{theorem}

\begin{proof}[Proof of Theorem \ref{MDP}]
We start with $x \geq 0$ and obtain
\begin{eqnarray*}
\big| \log \PP_n(W_n > a_n x) - \log \PP( Z > a_n x) \big| & = & \bigg| \log \biggl( 1 + {\mathcal O}(1) (1+ (a_n x)^3) \frac{1}{\sqrt{n}} \biggr) \bigg| \\
&\leq&  {\mathcal O}(1) (1+ (a_n x)^3) \frac{1}{\sqrt{n}}.
\end{eqnarray*}
For $x \geq 0$ it is known that
$$
\frac{1}{1 + \sqrt{2 \pi} x} \leq e^{x^2/2} \PP(Z > x) \leq \frac 12.
$$
The monotonicity of the logarithm implies
\begin{eqnarray*}
\bigg| \log \PP_n(W_n > a_n x) + \frac{(a_n x)^2}{2} \bigg| & \leq & \bigg| \log \bigl( e^{\frac{(a_n x)^2}{2}} P(Z > a_n x) \bigr) \bigg| +  {\mathcal O}(1) (1+ (a_n x)^3) \frac{1}{\sqrt{n}}\\
& \leq & \bigg| \log \biggl( \frac{1}{2 + \sqrt{2 \pi} a_n x} \biggr) \bigg| +  {\mathcal O}(1) (1+ (a_n x)^3) \frac{1}{\sqrt{n}}\\
& \leq & \log(2 + \sqrt{2 \pi} a_n x) +   {\mathcal O}(1) (1+ (a_n x)^3) \frac{1}{\sqrt{n}},
\end{eqnarray*}
and it follows by our assumption on $(a_n)_n$ that 
$$
\lim_{n \to \infty} \frac{1}{a_n^2} \log \PP_n( W_n > a_n x) = - \frac{x^2}{2}.
$$
It is easy to verify, that the Cram\'er-type moderate deviations implies
$$
\frac{\PP_n( W_n < - x)}{\PP (Z < - x)} = 1 +  {\mathcal O}(1)  (1+x^3)  \bigl( \frac{1}{\sqrt{n}} \bigr).
$$
Our bounds can be carried forward to a full moderate deviation principle similarly to the proof of Theorems 4.1.11 and 1.2.18 in \cite{DemboZeitouni:book}.
\end{proof}
\bigskip


 \end{document}